\documentclass[11pt]{amsart}
\usepackage{amssymb,amsmath}
\usepackage{eucal}
\usepackage{epstopdf}
\usepackage{mathrsfs}
\usepackage{enumerate}
\usepackage{fullpage} 
\usepackage{setspace}
\usepackage{xcolor}
\usepackage{dsfont}

\usepackage{todonotes}
\usepackage{hyperref}

\usepackage[OT2,T1]{fontenc}
\DeclareSymbolFont{cyrletters}{OT2}{wncyr}{m}{n}
\DeclareMathSymbol{\Sha}{\mathalpha}{cyrletters}{"58}

\newcommand{\X}{X}

\def\p{\mathfrak{p}}
\def\n{\mathfrak{n}}
\def\m{\mathfrak{m}}

\newcommand{\norm}{\mathrm{N}}

\newcommand{\oh}{\mbox{$\frac{1}{2}$}}

\DeclareFontFamily{OT1}{rsfs}{}
\DeclareFontShape{OT1}{rsfs}{n}{it}{<-> rsfs10}{}
\DeclareMathAlphabet{\mathscr}{OT1}{rsfs}{n}{it}

\newtheorem{prop}{Proposition}[section]

\newtheorem{theorem}[prop]{Theorem}

\newtheorem{lemma}[prop]{Lemma}

\newtheorem*{defn*}{Definition}

\numberwithin{equation}{section}

\renewcommand{\Re}{{\mathfrak{Re}}}

 \newcommand{\mymod}[1]{(\operatorname{mod} #1)}

\DeclareFontFamily{U}{mathx}{}
\DeclareFontShape{U}{mathx}{m}{n}{<-> mathx10}{}
\DeclareSymbolFont{mathx}{U}{mathx}{m}{n}
\DeclareMathAccent{\widehat}{0}{mathx}{"70}
\DeclareMathAccent{\widecheck}{0}{mathx}{"71}
\usepackage{microtype}

\theoremstyle{remark}
\newtheorem*{remark}{Remark}

\begin{document}

\title[Towards Keating-Snaith's conjecture for cubic Hecke $L$-functions]{Towards Keating-Snaith's conjecture for cubic Hecke $L$-functions over the Eisenstein field}

\author[H. Lin]{Hua Lin}
\address[Hua Lin]{Department of Mathematics\\
Northwestern University, Evanston, IL, USA}
\email{hua.lin@northwestern.edu}

\author[P.-J. Wong]{Peng-Jie Wong}
\address[Peng-Jie Wong]{Department of Applied Mathematics\\
National Sun Yat-Sen University\\
Kaohsiung City, Taiwan}
\email{pjwong@math.nsysu.edu.tw}


%  11M06  	$\zeta (s)$ and $L(s, \chi)$
%  11M41  	Other Dirichlet series and zeta functions 

%\date{\today}


\thanks{PJW is currently supported by the NSTC grant 114-2628-M-110-006-MY4.}

\begin{abstract}
    A famous conjecture of Keating and Snaith asserts that central values of $L$-functions in a given family admit a log-normal distribution with a prescribed mean and variance depending on the symmetry type of the family. Based on a recent work of {Radziwi\l\l} and Soundararajan, we obtain a conditional lower bound towards Keating-Snaith's conjecture for a ``thin'' family of cubic Hecke $L$-functions over the Eisenstein field. A key new input is certain twisted estimates of the 1-level density of zeros of cubic Hecke $L$-functions, extending the previous work of David and G\"{u}lo\u{g}lu, under the Generalised Riemann Hypothesis. 
\end{abstract}

\maketitle
%\tableofcontents

\section{Introduction}

{About a century ago, Hilbert and P\'{o}lya independently suggested a connection between zeros of the Riemann zeta function and eigenvalues of some self-adjoint operator. There was little evidence at the time, but since the encounter between Montgomery and Dyson about half a century ago, the Riemann zeta function and $L$-functions have been intimately linked to random matrix theory.}
Using this connection, Keating and Snaith conjectured that central values of $L$-functions in families admit a log-normal distribution with a prescribed mean and variance depending on their symmetry type \cite{KS}. Their far-reaching conjecture presents a natural analogue of Selberg's central limit theorem on the log-normality of the Riemann zeta function on the critical line $\Re(s)=\frac{1}{2}$ and lies in the centre of the theory of $L$-functions. For instance, Keating-Snaith's conjecture implies that almost all the central values of $L$-functions in families are non-vanishing, which is unknown in general.

In the same vein, prior to the work of Keating and Snaith, Katz and Sarnak \cite{KaSa, KaSa2}
%is this only one of their two papers?
analysed random matrices and monodromy for families of $L$-functions over function fields. Their work not only shed light on determining the symmetry type of families of $L$-functions, but also led to their famous density conjecture on the distribution of ``low-lying'' zeros of $L$-functions. 
Their conjecture, known as Katz and Sarnak's philosophy, implies the non-vanishing of families of $L$-functions at the central point. 
%does their conjecture have explicit non-vanishing %?
For example, some progress has been made for families of $L$-functions of modular forms by Iwaniec, Luo, and Sarnak \cite{ILS}, and for families of $L$-functions of quadratic twists of elliptic curves by Heath-Brown \cite{H-B}.

Most families that have been studied, however, 
are over the rationals $\Bbb{Q}$ and are ``self-dual''. Nonetheless, Cho and Park studied the low-lying zeros of cubic Dirichlet $L$-functions in \cite{CP} and recently, David and G\"{u}lo\u{g}lu \cite{DG} studied the 1-level density of zeros for cubic Hecke $L$-functions over the Eisenstein field $\mathcal{K}=\Bbb{Q}(\omega)$, where $\omega=e^{2\pi i/3}$, under the Generalised Riemann Hypothesis (GRH). For the ``thin'' family $\mathcal{F}$ defined in \eqref{def of thin family}, using the work of Heath-Brown \cite{H-B 2}, David and G\"{u}lo\u{g}lu obtained a support for the Fourier transform of the test function beyond the $(-1,1)$-barrier and achieved at least $\frac{2}{13}$ non-vanishing at the central point, supporting Katz and Sarnak's philosophy.

It is also worth mentioning that by calculating the moments of cubic Hecke $ L$-functions, it has been shown in \cite{BY,G} that there are infinitely many non-vanishing central values. 
Moreover, positive propositions of non-vanishing at the central point in the same family
%of cubic Hecke $ L$-functions 
were recently established in \cite{GY} (under GRH) and \cite{DdeFDS} (unconditionally). 
We refer the reader to these references for more details.

%\pj{\cite{GY}: modified moments: super small non-vanishing $\approx e^{-100}$\\ \cite{DdeFDS}: modified moments: $14\%$}

More recently, {Radziwi\l\l} and Soundararajan \cite{RaSo} proved a conditional lower bound towards Keating-Snaith's conjecture for quadratic twists of elliptic curves with global root number 1. This was extended by the second-named author to quadratic twists with global root number $-1$ in \cite{Wo}. The innovation of {Radziwi\l\l} and Soundararajan is a method of obtaining lower bounds on the non-vanishing of central $L$-values via studying ``twisted'' 1-level density estimates for low-lying zeros and using the ``method of moments'' from probability theory. As an illustration, they established conditional lower bounds towards Keating-Snaith's conjecture by ``upgrading'' the work of Heath-Brown \cite{H-B}.
%we will probably need to cite this to avoid confusion

In this article, we extend the method of Radziwi\l\l\ and Soundararajan \cite{RaSo} to the \emph{unitary} family $\mathcal{F}$ of cubic Hecke characters over $\Bbb{Z}[\omega]$, considered by David and G\"{u}lo\u{g}lu \cite{DG} (see also \eqref{def of thin family}), and prove the following theorem towards Keating-Snaith's conjecture (cf. \cite[pg. 59, (iii)]{KS-zeta}). 

\begin{theorem}\label{main-thm}
		\label{KSLB}
		Assume the Generalised Riemann Hypothesis (GRH) for $L(s, \chi_\mathfrak{f})$ with
        $\mathfrak{f}\in\mathcal{F}$. For any fixed 
        $(\alpha,\beta)$, we have as $ X \rightarrow \infty$,
		\begin{align*}
			\begin{split}
				\#\bigg\{ \mathfrak{f}\in\mathcal{F} &: X< \norm (\mathfrak{f})\le 2X,   \frac{\log |L(\frac{1}{2}, \chi_\mathfrak{f})|  }{\sqrt{ \log \log \norm(\mathfrak{f})}} \in (\alpha,\beta) \bigg\}\\
				&\ge \frac{2}{13}\left(\int_\alpha^\beta \frac{1}{\sqrt{ 2\pi}}e^{-\frac{1}{2} t^2} dt +o(1) \right)
				\# \{ \mathfrak{f} \in\mathcal{F} :  X< \norm(\mathfrak{f})\le 2X\},
			\end{split}
		\end{align*}
        where the integrand is the probability density function of a standard normal random variable.
    \end{theorem}
\noindent

\begin{remark}
(i) This is the first result of this type for a family of non-rational and non-self-dual $L$-functions to the best of the authors' knowledge.\\
(ii) Compared to the argument of David and G\"{u}lo\u{g}lu \cite{DG}, our analysis of the twisted 1-level density requires a more delicate treatment of twisted cubic Gauss sums where conductors are no longer prime powers. As seen in Section \ref{Proof of prop 3}, this is the core of our argument.
\end{remark}

\section{Notation and preliminaries}\label{background}

Let $\omega=e^{2\pi i/3}=\frac{-1+\sqrt{-3}}{2}$ be a primitive third root of unity, and let $\mathcal{K}=\mathbb{Q}(\omega)$. The ring of integer $\mathcal{O}_\mathcal{K}=\mathbb{Z}[\omega]$, where
every element $n$ can be expressed as $n=a+b\omega$ for some $a, b \in \mathbb{Z}.$
Since $\mathbb{Z}[\omega]$ is a principal ideal domain, every $n$ coprime to 3 has a unique generator $n=(m)$, $m\equiv1\mymod{3}.$ We use this fact freely throughout the article. Also, we let $\Lambda_\mathcal{K}, \phi_\mathcal{K},$ and $\mu_\mathcal{K}$ denote the number field analogues of the corresponding arithmetic functions (namely, $\Lambda(n), \phi(n),$ and $\mu(n)$) over $\mathcal{K}.$ 

We denote by $\norm(n)$ the norm of $n\in \mathbb{Z}[\omega]$, and $\norm(n)=n\bar{n}$, where $\bar{n}$ is the complex conjugate of $n.$ Furthermore, if $n=a+b\omega,$ then one has
\begin{align*}
    \norm(n)=(a+b\omega)(a+b\bar{\omega})=a^2+ab+b^2\in \mathbb{Z}.
\end{align*}

The cubic Dirichlet character $\chi_p$ modulo a prime $p\equiv 1\mymod{3}\in \mathbb{Z}[\omega]$ is the cubic residue symbol
\begin{align*}
    \chi_p(n)=\left(\frac{n}{p}\right)_3\equiv n^{\frac{\norm(p)-1}{3}}\mymod{p},
\end{align*}
and cubic characters of general moduli $n$ are defined multiplicatively. We recall the following statement for cubic reciprocity.
\begin{lemma}
    Let $m, n\in \mathbb{Z}[\omega]$ be coprime to $3$. If $m,n\equiv \pm 1\mymod{3}$, then
    \begin{align*}
        \left(\frac{m}{n}\right)_3=\left(\frac{n}{m}\right)_3.
    \end{align*}
\end{lemma} 

A character is primitive if it is the product of $\chi_p$ or $\overline{\chi_p}=\chi_{p^2}$ for distinct primes $p$. A cubic Dirichlet character is a cubic Hecke character if $\chi_n(\omega)=1.$ This is true if and only if $\norm(n)\equiv 1\mymod{9}$, so primitive Hecke characters have conductors $n=f_1f_2^2,$ where $f_1, f_2$ are square-free, with $(f_1, f_2)=1$ and  $\norm(f_1f_2^2)\equiv1\mymod{9}.$ The thin family $\mathcal{F}$ of cubic Hecke characters over $\Bbb{Z}[\omega]$ is, as considered by David and G\"{u}lo\u{g}lu in \cite{DG},
\begin{align}
    \label{def of thin family}
    \mathcal{F}=\{ \mathfrak{f}\in \mathbb{Z}[\omega] : \mathfrak{f}\ne1, \mathfrak{f}\equiv1\mymod{9}, \mathfrak{f} \text{ square-free} \}.
\end{align}
%We will sum over primitive characters $\chi_\mathfrak{f}$, with conductors $\mathfrak{f}\in\mathcal{F}.$

Let $\chi_\mathfrak{f}$ be a primitive Hecke character with $\mathfrak{f}\ne1$ and $(\mathfrak{f},3)=1.$ The completed $L$-function is 
\begin{align*}
    \Lambda(s, \chi_\mathfrak{f})=(|\Delta_\mathcal{K}|\norm(\mathfrak{f}))^{s/2}(2\pi)^{-s}\Gamma(s)L(s, \chi_\mathfrak{f}),
\end{align*}
where $\Delta_\mathcal{K}=-3$ is the discriminant of $\mathcal{K}=\mathbb{Q}(\omega)$. By the work of Hecke, it is known that $\Lambda(s, \chi_\mathfrak{f})$ is entire and it satisfies the functional equation
\begin{align*}
    \Lambda(s, \chi_\mathfrak{f})
    =
    \chi_\mathfrak{f}(\sqrt{-3})g(1, \mathfrak{f}) \norm(\mathfrak{f})^{-1/2}\Lambda(1-s, \overline{\chi_\mathfrak{f}}),
\end{align*}
where $g(1, \mathfrak{f})$ is the generalized cubic Gauss sum defined in \eqref{def cubic Gauss sum}.

Let $h$ be a smooth function with a compactly supported Fourier transform
$$
\widehat{h}(\xi)= \int_{\Bbb{R}} h(t) e^{-2\pi i\xi t}dt,
$$
and $h$ satisfies $|h(t)| \ll 1/(1+t^2)$ for all real $t$. In particular, one may take
$\displaystyle 
h(t) = \Big( \frac{\sin(\pi t)}{\pi t} \Big)^2,
$
the Fej\'er kernel, so that $\widehat{h}(\xi) =\max\{1 - |\xi|, 0\}$. In addition, let $\Phi$ be a smooth non-negative function compactly supported in $[\frac{1}{2},\frac{5}{2}]$ such that $\Phi(t)=1$ on $[1,2]$ and $\widehat{\Phi}(x)\ll_K x^{-K}$ for any $K>0$, where
\begin{align}\label{widehat Phi}
    \widehat{\Phi}(t)=\int_\mathbb{R}\int_\mathbb{R}\Phi(\norm(x+\omega y)e(-ty)dxdy.
\end{align}

Let $L \ge 1$ be a real number. The size of the log-derivative of $L(s, \chi_\mathfrak{f})$ is controlled on horizontal integrals, so by contour integration and Cauchy's theorem, one has the following formula. (See \cite[Lemmata 2.4-2.6]{DG} for a complete treatment.) 
\begin{lemma} [Explicit Formula]\label{explicit formula}
    Let $\chi_\mathfrak{f}$ be a primitive cubic Hecke character such that $(\mathfrak{f}, 3)=1$, and $h(x)$ be an even Schwartz function with compactly supported Fourier transform. Then
    \begin{align*}
        \sum_{\substack{\rho =\frac{1}{2}+i\gamma_\mathfrak{f}}} h\left( \frac{\gamma_\mathfrak{f}L}{2\pi} \right)
        =
        \frac{1}{2\pi}\int_{-\infty}^\infty h\left( \frac{tL}{2\pi} \right)
        \left( \log \frac{3\norm(\mathfrak{f})}{4\pi^2} + 2\frac{\Gamma^\prime}{\Gamma}\left(\frac{1}{2}+it\right) \right)dt
        \\
        -\frac{1}{L}\sum_{n\in \mathbb{Z}[\omega]}\frac{\Lambda_K(n)}{\sqrt{\norm(n)}}
        \left( \chi_\mathfrak{f}(n)+\overline{\chi_\mathfrak{f}(n)} \right) \hat{h}\left( \frac{\log\norm(n)}{L} \right).
    \end{align*}
\end{lemma}

A key step in estimating the sum over zeros (for example, in Theorem \ref{theorem S_2 chi}) is applying the following Poisson summation formula over $\mathbb{Z}[\omega]$. Throughout this article, we let
\begin{align}\label{def cubic exp}
    e(z)=\exp(2\pi i \mathrm{Tr}_{\mathbb{C}/\mathbb{R}}(z))=\exp(2\pi i (z+\bar{z})).
\end{align}

\begin{lemma}[Poisson Summation Formula]
    \label{lemma Poisson}
    Let $q, r\in \mathbb{Z}[\omega]$ be fixed. Then 
    $$\sum_{m\equiv r\mymod q} \Phi\left(\frac{\norm(m)}{M}\right)=\frac{M}{\norm(q)}\sum_{k\in \mathbb{Z}[\omega]}\widehat{\Phi}\left(\sqrt{\frac{M\norm(k)}{\norm(q)}} \right)e\left(\frac{-kr}{q\sqrt{-3}}\right),$$
    where $\widehat{\Phi}$ is defined in \eqref{widehat Phi}.
\end{lemma}
\noindent
This is the result of modifying the usual Poisson summation over $\mathbb{Z}[\omega]$ (for example, the one in \cite[Lemma 2.9]{DG}) following the procedure of \cite[Lemma 7]{RaSo15}).

We note that the sum over all residues for the exponential in \eqref{def cubic exp} still vanishes, since
\begin{align*}
    \sum_{r\mymod{p}}e\left(\frac{r}{p}\right)
    =
    \sum_{r\mymod{p}}\exp\left(\frac{(2\pi i)(r\bar{p}+\bar{r}p)}{\norm(p)}\right)
    =
    \sum_{k\mymod{\norm(p)}}\exp\left(\frac{2\pi ik}{\norm(p)}\right)
    =
    0,
\end{align*}
where for the second to last equality, we let
$k=r\bar{p}+\bar{r}p$ and observe that $k\in \mathbb{Z}$ with $k\mymod{\norm(p)}$ uniquely determined by $r\mymod{p}.$ 

The generalized cubic Gauss sum is 
\begin{align}\label{def cubic Gauss sum}    
g(k, n)=\sum_{r\mymod{n}}\chi_n(r)e\bigg(\frac{kr}{n}\bigg),
\end{align}
where $e(z)$ is defined in \eqref{def cubic exp}.
Twisted 1-level density led us to treat twisted cubic Gauss sums. We collect the technical lemmas regarding those below. 
\begin{lemma}
\label{lemma gauss sum with p^2 vanishes}
Let $p\in \mathbb{Z}[\omega]$ be a prime, $k, p\equiv 1\mymod{3}$, and $(k, p)=1.$ Then for all integers $\alpha\ge 2,$
    \begin{align*}
        g(k, p^\alpha)=\sum_{r\mymod{p^\alpha}}\chi_{p^\alpha}(r)e\left(\frac{kr}{p^\alpha}\right) = 0.
    \end{align*}
\end{lemma}
\begin{proof}
Dividing $r$ by $p^{\alpha-1}$, we have
    \begin{align*}
        g(k, p^\alpha)
        =
        \sum_{r\mymod{p^\alpha}}\chi_{p^\alpha}(r)e\left(\frac{rk}{p^\alpha}\right)
        =
        \sum_{r_1 \mymod{p^{\alpha-1}}}e\left(\frac{r_1k}{p^\alpha}\right)
         \sum_{r_2\mymod{p}}\chi_{p^\alpha}(r_1+r_2{p^{\alpha-1}})e\left(\frac{r_2k}{p}\right),
    \end{align*}
    where
    $$r=r_1+r_2{p^{\alpha-1}}, \ \text{ for some } r_1\mymod{p^{\alpha-1}}, \text{ and } r_2\mymod{p}.$$
    
    The character is nonzero when $$(r_1+r_2p^{\alpha-1}, p^\alpha)=1 \iff (r_1+r_2p^{\alpha-1}, p^{\alpha-1})=1 \iff (r_1, p^{\alpha-1})=1,$$ and let $\chi_p^*$ denote the primitive character inducing $\chi_{p^\alpha}$ or the principal character when $3\mid \alpha$. 
    Therefore,
    \begin{align*}
        g(k, p^\alpha)
        =
        \sum_{\substack{r_1 \mymod{p^{\alpha-1}}\\(r_1, p^{\alpha-1})=1}}e\left(\frac{r_1k}{p^\alpha}\right)\chi_p^*(k)
         \sum_{r_2\mymod{p}}e\left(\frac{r_2k}{p}\right),
    \end{align*}
    where since $(k, p)=1,$ the inner sum vanishes.
\end{proof}

We also require the following modified version of \cite[Lemma 2.8]{DG}.
\begin{lemma}
\label{lemma Gauss sum=0}
Let $p\in \mathbb{Z}[\omega]$ be a prime and $k, \ell, p\equiv 1\mymod{3}$. Suppose $p^a\mid \mid \ell$, $\ell=p^a\ell^\prime,$ $p^b\mid \mid k$, $k=p^bk^\prime$, and $j\ge 1$ is a rational integer. 
If $b=j+a-1$, then
    \begin{align*}
        g(k, p^j\ell)=
        \norm(p)^bg(k, \ell^\prime)
        \times
        \begin{cases}
            -1 &\text{ if } j+a \equiv 0 \mymod{3},\\
            \overline{\chi_p(\ell^\prime)} g(k^\prime, p) &\text{ if } j+a\equiv 1\mymod{3},\\
            \chi_p(\ell^\prime)\overline{g(k^\prime, p)} &\text{ if } j+a\equiv 2\mymod{3}.
        \end{cases}
    \end{align*}
If $b\ge j+a$ and $j+a \equiv 0\mymod{3},$ then
\begin{align*}
    g(k, p^j\ell)=\phi_{\mathcal{K}}(p^{j+a})g(k, \ell^\prime).
\end{align*}
In all other cases, $g(k, p^j\ell)=0.$
\end{lemma}

\begin{proof}
    Since $p^j\ell=p^{j+a}\ell^\prime$, where $(p, \ell^\prime)=1,$ we have
    \begin{align}
    \label{eq. gauss sum lemma eq 1}
        g(k, p^j\ell)=g(k, p^{j+a}\ell^\prime)
        =\chi_{p^{j+a}}(\ell^\prime)\chi_{\ell^\prime}(p^{j+a})
        g(k, p^{j+a})g(k, \ell^\prime).
    \end{align}
    When $b=j+a-1$, one has
    \begin{align}
    \label{eq. DG 2.8}
        g(k, p^{j+a})=
        \norm(p^b)
        \times
        \begin{cases}
            -1 &\text{ if } j+a \equiv 0 \mymod{3},\\
             g(k^\prime, p) &\text{ if } j+a\equiv 1\mymod{3},\\
            \overline{g(k^\prime, p)} &\text{ if } j+a\equiv 2\mymod{3},
        \end{cases}
    \end{align}
    by \cite[Lemma 2.8]{DG}.
    Using cubic reciprocity, we have 
    \begin{align*}
    \chi_{p^{j+a}}(\ell^\prime)\chi_{\ell^\prime}(p^{j+a})
    =
    \begin{cases}
        1 &\text{if } j+a \equiv 0\mymod{3},\\
        \overline{\chi_p(\ell^\prime)} &\text{if } j+a \equiv 1\mymod{3},\\
        \chi_p(\ell^\prime) &\text{if } j+a \equiv 2\mymod{3}. 
    \end{cases}   
    \end{align*}
    Thus, by \eqref{eq. gauss sum lemma eq 1} and \eqref{eq. DG 2.8}, we have shown the statement in the lemma when $b=j+a-1.$

    Now if $p^{j+a} \mid k,$ which is the case when $b\ge j+a,$ $g(k, p^{j+a})$ is nonzero only if $j+a\equiv 0\mymod{3}$. In that case, the first three terms in \eqref{eq. gauss sum lemma eq 1} give $\phi_\mathcal{K}(p^{j+a})$. 

    In the remaining cases, $0\le b\le j+a-2$, and it suffices to analyze $g(k^\prime, p^{j+a-b})$. By Lemma \ref{lemma gauss sum with p^2 vanishes}, $g(k, p^{j+a})$ vanishes.
\end{proof}

In light of the work of {Radziwi\l\l} and Soundararajan \cite{RaSo}, we prove below a ``twisted-counting'' for the number of $\mathfrak{f}\in \mathcal{F}$. (Note that when $\ell=1$, it gives a count of the number of $\mathfrak{f}\in \mathcal{F}$ with $\norm(\mathfrak{f}) \asymp X$.) Let $\zeta_\mathcal{K}(s)$ denote the Dedekind zeta function of  $\mathcal{K}=\mathbb{Q}(\omega).$

\begin{theorem}
\label{theorem S_1}
    Let $\ell\in\mathbb{Z}[\omega]$, such that $\norm(\ell)^{1/4}\le \sqrt{X}$, and let
    $$S=\sum_{\mathfrak{f}\in \mathcal{F}}\chi_{\mathfrak{f}}(\ell)\Phi\left(\frac{\norm(\mathfrak{f})}{X}\right),$$
    where $\Phi$ is a smooth function defined as above.   
    When $\ell$ is a cube, for any $A>0,$ we have
    $$S=\frac{X\widehat{\Phi}(0)}{\norm(9)}  \zeta^{-1}_\mathcal{K}(2)  \prod_{p\mid \ell}\left(1+\frac{1}{\norm(p)}\right)^{-1} +O\left(\frac{X}{A}\right)+O\left(\frac{A}{X^{1+\epsilon}}\right).$$
    When $\ell$ is not a cube, we have
    $$S
    \ll
    \frac{X}{A}+\frac{A}{X^{1+\epsilon}}.
    $$
Consequently, with the choice of $A=X^{1/2}/\norm(\ell)^{1/4}$, all the errors above are
$\ll X^{1/2}\norm(\ell)^{1/4}$.
\end{theorem}

\begin{proof} For brevity, let $|\cdot|$ denote the ideal norm in the following proof. Using 
\begin{equation}\label{detect-square}
\sum_{\substack{d\in \mathbb{Z}[\omega]\\d^2\mid \mathfrak{f}}}\mu_\mathcal{K}(d)=\begin{cases}
    1 &\text{ if } \mathfrak{f} \text{ is square-free,}\\
    0 &\text{ otherwise,}
\end{cases}
\end{equation}
to detect square-free $\mathfrak{f}$, and writing $\mathfrak{f}=d^2f$, we have
    \begin{align*}
    S
%    =
%    \sum_{\mathfrak{f}\in \mathcal{F}}\chi_{\mathfrak{f}}(\ell)\Phi\left(\frac{|\mathfrak{f}|}{X}\right)
    =
    \sum_{\substack{d\in \mathbb{Z}[\omega]}}\mu_\mathcal{K}(d)\chi_{d^2}(\ell)
    \sum_{f \equiv \overline{d^2}\mymod9}\chi_{f}(\ell)\Phi\left(\frac{|fd^2|}{X}\right).
\end{align*}

We trivially bound the sum for $|d|>A$ as
\begin{align*}
    \sum_{\substack{|d|>A}}\mu_\mathcal{K}(d)\chi_{d^2}(\ell)
    \sum_{f \equiv \overline{d^2}\mymod9}\chi_{f}(\ell)\Phi\left(\frac{|fd^2|}{X}\right)
    \ll
    \sum_{\substack{|d|>A}}\frac{X}{|d^2|}
    \ll
    \frac{X}{A}.
\end{align*}
For $|d|\le A$, after applying Poisson summation to the sum over $f$,
\begin{align}
\begin{split}
    &\sum_{\substack{|d|\le A}}\mu_\mathcal{K}(d)\chi_{d^2}(\ell)
    \sum_{f \equiv \overline{d^2}\mymod9}\chi_{f}(\ell)\Phi\left(\frac{|fd^2|}{X}\right)
    \\
    &=
    \frac{X}{|9|}\sum_{\substack{|d|\le A}}\frac{\mu_\mathcal{K}(d)\chi_{d^2}(\ell)}{|d^2|}
    \sum_{k\in \mathbb{Z}[\omega]} \widehat{\Phi}\bigg(\sqrt{\frac{X|k|}{|9{\ell }d^2|}}\bigg)
    \sum_{r\mymod {\ell }} \frac{\chi_{\ell }(r)}{|\ell|}e\bigg(\frac{-( r9\overline{9}+\overline{d^2}\ell\overline{\ell})k}{{9\ell }\sqrt{-3}}\bigg).
\end{split}
\end{align}
When $k=0$, the above contribution is nonzero if and only if $\ell$ is a cube, and in that case
$$S=\frac{X\widehat{\Phi}(0)}{|9|}  \zeta^{-1}_\mathcal{K}(2)  \prod_{p\mid \ell}\left(1+\frac{1}{|p|}\right)^{-1} +O\left(\frac{X}{A}\right).$$
When $k\ne 0,$ we have
\begin{align*}
    &\frac{X}{|9|}\sum_{\substack{|d|\le A\\3\nmid d}}\frac{\mu_\mathcal{K}(d)\chi_{d^2}(\ell)}{|d^2|}
    \sum_{\substack{k\in \mathbb{Z}[\omega]\\k\ne 0}} \widehat{\Phi}\bigg(\sqrt{\frac{X|k|}{|9{\ell }d^2|}}\bigg)
    e\bigg(\frac{- \overline{d^2\ell}k}{{9 }\sqrt{-3}}\bigg)
    \frac{1}{|\ell|}\sum_{r\mymod {\ell }} \chi_{\ell }(r)e\left(\frac{- r\overline{9}k}{{\ell }\sqrt{-3}}\right)\\
    &\ll
    \sum_{\substack{|d|\le A}}\frac{X}{|9d^2|} \frac{|9\ell d^2|}{X} \left(\frac{|9\ell d^2|}{X}\right)^{K-1}
    \sum_{\substack{k\in \mathbb{Z}[\omega]\\k\ne 0}}
    \left(\frac{1}{|k|}\right)^K
    \ll
    \frac{|9|^{K-1}|\ell|}{X^{\epsilon K}}\sum_{\substack{|d|\le A}} 
    \sum_{\substack{k\in \mathbb{Z}[\omega]\\k\ne 0}}
    \left(\frac{1}{|k|}\right)^K\ll
    \frac{|\ell|A}{X},
\end{align*}
where since
$$\frac{X}{|\ell d^2|} \ge X^\epsilon \text{ for some } \epsilon<1, \text{ and } \  \widehat{\Phi}(x)\ll_K x^{-K} \text{ for any } K>0,$$
we choose $K>1$ such that $\epsilon K=1.$
Thus, when $\ell$ is a cube, we have 
$$S=\frac{X\widehat{\Phi}(0)}{|9|}  \zeta^{-1}_\mathcal{K}(2)  \prod_{p\mid \ell}\left(1+\frac{1}{|p|}\right)^{-1} +O\left(\frac{X}{A}\right)+O\left(\frac{|\ell|A}{X}\right),$$
and when $\ell$ is not a cube, 
$$S
\ll
\frac{X}{A}+\frac{|\ell|A}{X},$$
which completes the proof.
\end{proof}

Finally, we recall some facts regarding normal random variables. Let $\mathcal{N}$ be a normal random variable with mean 0 and variance $\sigma^{2}$. The $k^\mathrm{th}$ moment of $\mathcal{N}$ is given by 
\begin{align*}
\mathbb{E}[\mathcal{N}^{k}]
=\begin{cases} 0&\mbox{if $k$ is odd;}\\
(k-1)!!\sigma^{n}&\mbox{if $k$ is even,}
\end{cases}
\end{align*}
where $\mathbb{E}[\mathcal{N}^{k}]$ denotes the  expectation of $\mathcal{N}^{k}$, and  $k!!$ stands for the double factorial defined by
\begin{align}\label{def-!!}
k!!:=\prod_{j=0}^{\lceil k/2\rceil -1}(k-2j).
\end{align}

We shall require the following version of the ``method of moments'' due to Fr\'echet and Shohat \cite{FS}. 

\begin{lemma}
[Fr\'echet and Shohat]\label{probability}
Suppose that the distribution of a random variable $\mathcal{X}$ is determined by its moments and that each $\mathcal{X}_{n}$ has moments of all orders. If $\mathbb{E}(\mathcal{X}^{k})=\mathbb{E}(\mathcal{X}_{n}^{k})+o(\mathbb{E}(\mathcal{X}_{n}^{k}))$, as $n\rightarrow\infty$, for all $k\in \Bbb{N}$, then $\mathcal{X}_{n}$ converges to $\mathcal{X}$ in distribution.
\end{lemma}

\section{Key propositions}
We collect the statements of key propositions used to prove Theorem \ref{main-thm} in this section. (Their proofs, which are all under GRH, are given in Section \ref{section proofs of props}.) We begin with the following proposition, which allows us to approximate the central values of $ L$-functions by the real parts of certain Dirichlet polynomials.

\begin{prop}\label{Prop 1}
Let $\chi_\mathfrak{f}$ be a Hecke character of conductor $\mathfrak{f}$ over a number field $F$, and let $ 3\le x \le \norm( \mathfrak{f})$. Assume
GRH for $L(s,\chi_\mathfrak{f})$, and let $\frac{1}{2}+ i\gamma_\mathfrak{f} $ denote 
the non-trivial zeros of $ L(s,\chi_\mathfrak{f})$.  If  $L(\frac{1}{2}, \chi_\mathfrak{f})$ is non-vanishing, then
\begin{align}\label{exp-logL}
\log | L(\oh , \chi_\mathfrak{f})|
=  \Re \mathcal{P}(\chi_\mathfrak{f}; x) +\Delta(\chi_\mathfrak{f}) \log\log x +O(1)
+O_F\bigg( \frac{ \log \norm( \mathfrak{f})}{\log x}
+\sum_{\gamma_\mathfrak{f} } 
\log \bigg(1 + \frac{ 1 }{ (\gamma_\mathfrak{f} \log x)^2} \bigg)  \bigg),
\end{align}
where 
\begin{align}
\label{def Dirichlet polynomial}
    \mathcal{P}(\chi_\mathfrak{f}; x) = \sum_{p\le x} \frac{ \chi_\mathfrak{f}(\p)}{\norm(\p)^{\frac{1}{2}}} w(\p),
\quad
w(\p) = \frac{1}{\norm(\p)^{\frac{1}{\log x}}}   \frac{\log (x/\norm(\p))}{\log x}, 
\quad\text{and}\quad
\Delta(\chi_\mathfrak{f}) =
\begin{cases}
    \frac{1}{2} &\text{ if } \chi_\mathfrak{f}^2 =1,\\
    0 &\text{ otherwise.}
\end{cases}
\end{align}
\end{prop}

In light of this proposition and the method of moments (Lemma \ref{probability}), we calculate the following moments of $\mathcal{P}(\chi_\mathfrak{f}; x)$ to study the distribution of $\log | L(\oh , \chi_\mathfrak{f})|$. Indeed, as argued in \cite[Sec. 5]{HW}, such calculations allow one to deduce moments of $\Re \mathcal{P}(\chi_\mathfrak{f}; x) $ (see Lemma \ref{moments}).

\begin{prop}\label{Prop 2}
%Assume GRH. 
Let $k, j$ be non-negative integers and  $x=X^{\frac{13/22}{\log \log \log X}}$ for large  $X$. 
%and with $x^{\frac{3k+4j}{2}}\leq X$ 
Then (unconditionally)
\begin{align*}
 \sum_{\mathfrak{f}\in\mathcal{F}}  \mathcal{P}(\chi_\mathfrak{f}; x)^{k}\overline{\mathcal{P}(\chi_\mathfrak{f}; x)^{j}} \Phi\left(\frac{\norm(\mathfrak{f})}{X}\right)
 \ll_{k, j}
%x^{\frac{3k+4j}{4}-\epsilon}
X^{\frac{1}{2}+\epsilon},
\end{align*}
whenever $k\neq j$. If $k=j$,  we have (unconditionally)
\begin{align*}
 \sum_{\chi\in\mathcal{F}} \big|\mathcal{P}(\chi_\mathfrak{f}; x)\big|^{2k} \Phi\left(\frac{\norm(\mathfrak{f})}{X}\right)
 =\frac{k!X\widehat{\Phi}(0)}{81\zeta_\mathcal{K}(2)}  (\log\log X)^k+ O_{k}(X(\log\log X)^{k-1+\varepsilon}).
\end{align*}
Moreover, for any $L$ such that $e^{L}\le X^{13/11}$,
\begin{align*}
\sum_{\mathfrak{f}\in\mathcal{F}}  \mathcal{P}(\chi_\mathfrak{f}; x)^{k}\overline{\mathcal{P}(\chi_\mathfrak{f}; x)^{j}}\sum_{\gamma_{_\mathfrak{f}}} h\left( \frac{\gamma_{_\mathfrak{f}} L}
{2\pi}\right)\Phi\left(\frac{\norm(\mathfrak{f})}{X}\right)
\ll_{k, j}
%{(\max\{k, j\})!}
X \left(\log \log X\right)^{\min\{k,j\}}
%x^{\frac{55k+83j}{54}-\epsilon}
%+e^{\frac{11L}{27}}X^{\frac{14}{27}+\epsilon}
\end{align*}
whenever $k\neq j$. If $k=j$,  we have
\begin{align*}
 \sum_{\mathfrak{f}\in\mathcal{F}} \big|\mathcal{P}(\chi_\mathfrak{f}; x)\big|^{2k}\sum_{\gamma_{\mathfrak{f}}} h\left( \frac{\gamma_{\mathfrak{f}} L}{2\pi}\right)\Phi\left(\frac{\norm(\mathfrak{f})}{X}\right)
 =
 \frac{k!X\widehat{\Phi}(0)}{81\zeta_\mathcal{K}(2)L}     
\left(\hat{h}(0) \log X + O(1)\right) (\log\log X)^k\\
+ 
O_{k}(X(\log\log X)^{k-1+\varepsilon}).
\end{align*}
\end{prop}

As seen in Section \ref{pf-prop 2}, the moments calculations require, in a subtle way, the following ``twisted'' 1-level density estimates.

\begin{prop}\label{Prop 3}
%Assume GRH.
    Let $\ell\in \mathbb{Z}[\omega]$, $L\ge 1$ be a real number, and $h$ and $\Phi$ be defined as in Section \ref{background}. Set
    \begin{align}\label{eq. the c constant}
    C_{3, h}=\sum_{\substack{q\in \mathbb{Z}[\omega] }}
    \frac{\Lambda_\mathcal{K}(q)}{\sqrt{\norm(q)^3}}
    \left(1+\frac{1}{\norm(p)}\right)^{-1}   
	\hat{h}\left(\frac{3\log \norm(q)}{L}\right), 
\end{align}
and
$$
D^T(X;\ell,h, \Phi)=\sum_{\mathfrak{f}\in\mathcal{F}} \sum_{\gamma_{\mathfrak{f}}} h\left( \frac{\gamma_{\mathfrak{f}} L}{2\pi}\right)\chi_\mathfrak{f}(\ell) \Phi\left(\frac{\norm(\mathfrak{f})}{X}\right).
$$
    If $e^{\frac{11L}{27}}\norm(\ell)^{\frac{14}{27}}\le X^{\frac{13}{27}}$
(i.e., $e^{11L}\norm(\ell)^{14}\le X^{13}$), then when $\ell$ is a cube,
\begin{align*}
D^T(X;\ell,h, \Phi)=
 \frac{X\widehat{\Phi}(0)}{81L}  \zeta^{-1}_\mathcal{K}(2)  \prod_{p\mid \ell}\left(1+\frac{1}{\norm(p)}\right)^{-1} 
 \left(\hat{h}(0) \log X + C_{3,h}+ O(1)\right) 
 + O\bigg( \frac{X^{\frac{14}{27}+\epsilon}e^{\frac{11L}{27}}\norm(\ell)^{\frac{14}{27}}}{L} \bigg).
\end{align*}
When $\ell=qa^3$ a unique prime $q$ times a cube, then
\begin{align*}
D^T(X;\ell,h, \Phi)
\ll
     \frac{X\widehat{\Phi}(0)\log\norm(q)(1+\sqrt{\norm(q)})\zeta^{-1}_\mathcal{K}(2)  \prod_{p\mid q \ell }(1+\frac{1}{\norm(p)})^{-1}}{81(\norm(q)-\norm(q)^{-1/2})L} 
    + 
    \frac{X^{\frac{14}{27}+\epsilon}e^{\frac{11L}{27}}\norm(\ell)^{\frac{14}{27}}}{L}.
\end{align*}
    When neither is the case, 
    $$
    D^T(X;\ell,h, \Phi)
    \ll
    \frac{X^{\frac{14}{27}+\epsilon}e^{\frac{11L}{27}}\norm(\ell)^{\frac{14}{27}}}{L}.$$
\end{prop}

\section{Proof of the main theorem}

With Proposition \ref{Prop 1} in mind, we first derive the moments of $\Re \mathcal{P}(\chi_\mathfrak{f}; x) $ from Proposition \ref{Prop 2} as follows.

\begin{lemma}\label{moments} 
    Let $k$ be a positive integer and let $n!!$ denote the double factorial function defined as in \eqref{def-!!}. When $k$ is odd,
    \begin{align*}
        \sum_{\mathfrak{f}\in\mathcal{F}}\big(\Re(\mathcal{P}(\chi_\mathfrak{f}; x))\big)^{k}\Phi\left(\frac{\norm(\mathfrak{f})}{X}\right)
        \ll_k
        X^{\frac{1}{2}+\epsilon},
    \end{align*}
    and under GRH,
    \begin{align*}
        \sum_{\mathfrak{f}\in\mathcal{F}}\big(\Re(\mathcal{P}(\chi_\mathfrak{f}; x))\big)^{k}\sum_{\gamma_{\mathfrak{f}}} h\left( \frac{\gamma_{\mathfrak{f}} L}{2\pi}\right)\Phi\left(\frac{\norm(\mathfrak{f})}{X}\right)
        \ll_k
         X(\log \log X)^{\frac{k-1}{2}+\epsilon}.
    \end{align*}
    When $k$ is even
    \begin{align*}
        \sum_{\mathfrak{f}\in\mathcal{F}}\big(\Re(\mathcal{P}(\chi_\mathfrak{f}; x))\big)^{k}\Phi\left(\frac{\norm(\mathfrak{f})}{X}\right)
        =
        \Big(\frac{k}{2}-1\Big)!! \frac{X\widehat{\Phi}(0)}{81\zeta_\mathcal{K}(2)} \Big(\frac{1}{2}\log\log X\Big)^{\frac{k}{2}}+O_{k}\big(X(\log\log X)^{\frac{k}{2}-1+\epsilon}\big),
    \end{align*}
    and under GRH,
    \begin{align*}
        &\sum_{\mathfrak{f}\in\mathcal{F}}\big(\Re(\mathcal{P}(\chi_\mathfrak{f}; x))\big)^{k}\sum_{\gamma_{\mathfrak{f}}} h\left( \frac{\gamma_{\mathfrak{f}} L}{2\pi}\right)\Phi\left(\frac{\norm(\mathfrak{f})}{X}\right)\\
    &=
    \Big(\frac{k}{2}-1\Big)!!\frac{X\widehat{\Phi}(0)}{81\zeta_\mathcal{K}(2)L}     
    \left(\hat{h}(0) \log X + O(1)\right) \Big(\frac{1}{2}\log \log X\Big)^{\frac{k}{2}}
    +
    O( X(\log \log X)^{\frac{k}{2}-1+\epsilon} ).
    \end{align*}
\end{lemma}
\begin{proof}
Writing $\Re(\mathcal{P}(\chi_\mathfrak{f}; x)= \mathcal{P}(\chi_\mathfrak{f}; x)+\overline{\mathcal{P}(\chi_\mathfrak{f}; x)}$, by the binomial theorem, we have
\begin{align}\label{eq. real pjpk}
 \begin{split}
\sum_{\mathfrak{f}\in\mathcal{F}}\big(\Re(\mathcal{P}(\chi_\mathfrak{f}; x))\big)^{k}\Phi\left(\frac{\norm(\mathfrak{f})}{X}\right)
=
\frac{1}{2^{k}}\sum_{i=0}^{k}{k \choose i}\sum_{\mathfrak{f}\in\mathcal{F}}\mathcal{P}(\chi_\mathfrak{f}; x)^{i}\overline{\mathcal{P}(\chi_\mathfrak{f}; x)^{k-i}}\Phi\left(\frac{\norm(\mathfrak{f})}{X}\right),
\end{split}
\end{align}
and similarly,
\begin{align}\label{eq. real pjpk with h}     \begin{split}
    &\sum_{\mathfrak{f}\in\mathcal{F}}\big(\Re(\mathcal{P}(\chi_\mathfrak{f}; x))\big)^{k}\sum_{\gamma_{\mathfrak{f}}} h\left( \frac{\gamma_{\mathfrak{f}} L}{2\pi}\right)\Phi\left(\frac{\norm(\mathfrak{f})}{X}\right)\\
    &=
    \frac{1}{2^{k}}\sum_{i=0}^{k}{k \choose i}\sum_{\mathfrak{f}\in\mathcal{F}}\mathcal{P}(\chi_\mathfrak{f}; x)^i\overline{\mathcal{P}(\chi_\mathfrak{f}; x)^{k-i}}
    \sum_{\gamma_{\mathfrak{f}}} h\left( \frac{\gamma_{\mathfrak{f}} L}{2\pi}\right)
    \Phi\left(\frac{\norm(\mathfrak{f})}{X}\right).
\end{split}
\end{align}

According to Proposition \ref{Prop 2}, for any odd $k$ (so that it is impossible to have $i= k-i$), \eqref{eq. real pjpk} is bounded by 
$
\ll_{k}  
%\frac{1}{2^{k}}\sum_{j=0}^{k}{k \choose j} 
X^{\frac{1}{2}+\epsilon}.
$
If $k$ is even, the term with $i=k-i=k/2$ dominates, and we thus apply Proposition \ref{Prop 2} to obtain 
\begin{align*}
\sum_{\mathfrak{f}\in\mathcal{F}}\big(\Re(\mathcal{P}(\chi_\mathfrak{f}; x))\big)^{k}\Phi\left(\frac{\norm(\mathfrak{f})}{X}\right)
&=2^{-k} {{k}\choose{k/2}}\Big(\frac{k}{2}\Big)! \frac{X\widehat{\Phi}(0)}{81\zeta_\mathcal{K}(2)}  (\log\log X)^{\frac{k}{2}}+O_{k}\big(X(\log\log X)^{\frac{k}{2}-1+\epsilon}\big)\\
&=\Big(\frac{k}{2}-1\Big)!! \frac{X\widehat{\Phi}(0)}{81\zeta_\mathcal{K}(2)} \Big(\frac{1}{2}\log\log X\Big)^{\frac{k}{2}}+O_{k}\big(X(\log\log X)^{\frac{k}{2}-1+\epsilon}\big).
\end{align*}

Similarly, for $e^{11L/27}\le X^{13/27}$, \eqref{eq. real pjpk with h} for odd $k$ is bounded by
$\ll_k X(\log \log X)^{\frac{k-1}{2}+\epsilon}$. When $k$ is even, again $i=k-i=k/2$ dominates, so
\begin{align*}
    &\sum_{\mathfrak{f}\in\mathcal{F}}\big(\Re(\mathcal{P}(\chi_\mathfrak{f}; x))\big)^{k}\sum_{\gamma_{\mathfrak{f}}} h\left( \frac{\gamma_{\mathfrak{f}} L}{2\pi}\right)\Phi\left(\frac{\norm(\mathfrak{f})}{X}\right)\\
    &=
    2^{-k} {k\choose k/2} {\left(\frac{k}{2}\right)!}
    \frac{X\widehat{\Phi}(0)}{81\zeta_\mathcal{K}(2)L}     
    \left(\hat{h}(0) \log X + O(1)\right) 
    (\log \log X)^{k/2} + O\big( X(\log \log X)^{\frac{k}{2}-1+\epsilon}\big)
\end{align*}
as desired.
\end{proof}

Following \cite{RaSo}, we require the following lemma to control the contribution of zeros away from $\frac{1}{2}$ (by discarding certain $\mathfrak{f}\in \mathcal{F}$ with a large zero sum).

\begin{lemma}\label{RS-lemma2}
   Assume GRH, and let $x=X^{\frac{13/22}{\log \log\log X}}$.
    Then we have
    \begin{align*}
        \# \Big\{ X< \norm(\mathfrak{f})\le 2X: \sum_{(\log X\log \log X)^{-1}\le |\gamma_\mathfrak{f}|}
        \log\left( 1+\frac{1}{(\gamma_\mathfrak{f} \log x)^2} \right)
        \ge
        (\log \log \log X)^3 \Big\}
        \ll
        \frac{X}{\log \log \log X}.
    \end{align*}
\end{lemma} 

\begin{proof}
    By Proposition \ref{Prop 3}, we have 
    \begin{align}\label{eq. fejer kernel bound 1}
        \sum_{\substack{\mathfrak{f}\in \mathcal{F}\\X\le \norm(\mathfrak{f})\le 2X}} \sum_{\gamma_\mathfrak{f}}\left( \frac{\sin(\gamma_\mathfrak{f}L/2)}{\gamma_\mathfrak{f}L/2}\right)^2
        \ll
        \frac{X\log X}{L}.
    \end{align}
    Integrate $L$ in the range $\log x \le L \le 2\log x$ and using the bound
    \begin{align*}
        \frac{1}{y}\int_y^{2y} \left( \frac{\sin(\pi tu)}{\pi tu} \right)^2 du \gg
        \min\left\{ 1, \frac{1}{(ty)^2} \right\}
    \end{align*}
    for any $y>0$ and $t\ne 0$, we obtain (with $ t=\frac{\gamma_\mathfrak{f}}{2\pi}$)
    \begin{align*}
        \sum_{\substack{\mathfrak{f}\in \mathcal{F}\\X< \norm(\mathfrak{f})\le 2X}} \sum_{\gamma_\mathfrak{f}}
        \frac{1}{\log x} \int_{\log x}^{2\log x}
        \left( \frac{\sin(\gamma_\mathfrak{f}L/2)}{\gamma_\mathfrak{f}L/2}\right)^2 dL
        \gg
        \sum_{\substack{\mathfrak{f}\in \mathcal{F}\\X< \norm(\mathfrak{f})\le 2X}} \sum_{\gamma_\mathfrak{f}}
        \min\left\{ 1, \frac{(2\pi)^2}{(\gamma_\mathfrak{f}\log x)^2} \right\}.
    \end{align*} 
    The right side of \eqref{eq. fejer kernel bound 1} gives
    \begin{align*}
        \frac{X\log X}{\log x}\int_{\log x}^{2\log x}\frac{1}{L}dL
        \ll
        \frac{X\log X}{\log x}=\frac{22}{13}{X \log \log \log X},
    \end{align*}
    so
    \begin{align*}
        \sum_{\substack{\mathfrak{f}\in \mathcal{F}\\X< \norm(\mathfrak{f})\le 2X}} \sum_{\gamma_\mathfrak{f}}
        \min\left\{ 1, \frac{1}{(\gamma_\mathfrak{f}\log x)^2} \right\}
        \ll
        X\log \log \log X.
    \end{align*}
    Since for $|\gamma_\mathfrak{f}|\ge (\log X\log \log X)^{-1}$, one has
    \begin{align*}
        \log \left( 1+\frac{1}{(\gamma_\mathfrak{f} \log x)^2} \right)
        \ll
        (\log \log \log X) \min\left\{ 1, \frac{1}{(\gamma_\mathfrak{f}\log x)^2} \right\},
    \end{align*}
and thus
    \begin{align*}
        \sum_{\substack{\mathfrak{f}\in \mathcal{F}\\X<\norm(\mathfrak{f})\le 2X}}
        \sum_{(\log X\log \log X)^{-1}\le |\gamma_\mathfrak{f}|}
        \log\left( 1+\frac{1}{(\gamma_\mathfrak{f} \log x)^2} \right)
        \ll 
        X(\log \log \log X)^2,
    \end{align*}
which yields the claimed estimate.
\end{proof}

In what follows, to ease the notation, we denote
\begin{align*}
\Psi(\alpha,\beta) 
=
\int_\alpha^\beta \frac{1}{\sqrt{ 2\pi}}e^{-\frac{1}{2} t^2} dt,
\end{align*}
where the integrand is the probability density function of a standard normal random variable. 
To ensure that a positive proportion of the family $\mathcal{F}$ does not have low-lying zeros, we prove the following lower bound. 

\begin{lemma}\label{dist-lemma} 
Let $\alpha<\beta$ be real numbers. Let $\mathcal{H}_X(\alpha,\beta)$ be the set of discriminants $\mathfrak{f} \in\mathcal{F}$, with  $X < \norm(\mathfrak{f}) \le 2X$, such that
$$
\mathcal{Q}(\mathfrak{f};X) := \frac{\Re \mathcal{P}(\chi_\mathfrak{f}; x)   }{\sqrt{ \log \log X}}\in (\alpha,\beta) , 
$$
and $ L(s,\chi_{\mathfrak{f}})$ has no zeros $\frac{1}{2}+i\gamma_\mathfrak{f}$ with
$
 |\gamma_\mathfrak{f}|  \le( (\log X) (\log\log X))^{-1}.
$
Then under GRH, for any $\delta>0$, 
$$
\mathcal{H}_X(\alpha,\beta) 
\ge
\left(\frac{2}{13}-\delta\right)
( \Psi(\alpha,\beta) +o(1) )
\# \{ \mathfrak{f} \in\mathcal{F} :  X < \norm(\mathfrak{f}) \le 2X \}.
$$
\end{lemma}

\begin{proof}
Let $\Phi$ be a smooth approximation to the indicator function of $[1, 2]$, and let
$
h(t) = ( \frac{\sin(\pi t)}{\pi t} )^2
$ be 
the Fej\'er kernel.
From Lemma \ref{moments} and the method of moments (Lemma \ref{probability}), it follows that
\begin{equation}\label{1st-asym}
 \sum_{\substack{ \mathfrak{f} \in\mathcal{F}\\ \mathcal{Q}(\mathfrak{f};X) \in (\alpha,\beta) }}  \Phi\left(\frac{\norm(n)}{X}\right)
= ( \Psi(\alpha,\beta) +o(1) )  \sum_{\mathfrak{f} \in\mathcal{F}} \Phi\left(\frac{\norm(n)}{X}\right)
\end{equation}
and
\begin{align}\label{2nd-asym}
\begin{split}
 \sum_{\substack{ \mathfrak{f} \in\mathcal{F}\\ \mathcal{Q}(\mathfrak{f};X) \in (\alpha,\beta) }}  \sum_{\gamma_{\mathfrak{f} }}h\left( \frac{\gamma_{\mathfrak{f} } L}{2\pi}\right)\Phi\left(\frac{\norm(n)}{X}\right)
&= ( \Psi(\alpha,\beta) +o(1) )  \sum_{\mathfrak{f} \in\mathcal{F}} \sum_{\gamma_{\mathfrak{f} }}h\left( \frac{\gamma_{\mathfrak{f} } L}{2\pi}\right)\Phi\left(\frac{\norm(n)}{X}\right)\\
&= ( \Psi(\alpha,\beta)   +o(1) )
\frac{1}{L} 
    \left(\hat{h}(0) \log X + O(1)\right)\sum_{\mathfrak{f} \in\mathcal{F}} \Phi\left(\frac{\norm(n)}{X}\right) .
 \end{split}
\end{align}

Now, we take $L = (F-\eta) \log X$, with $F=\frac{13}{11}$, and we recall that $\hat{h}(0)=1$.  Denote $\mathcal{Z}$ the collection of $\mathfrak{f} \in\mathcal{F}$ such that  $L(s,\chi_\mathfrak{f})$ has no zeros with  $|\gamma_{\mathfrak{f} }| \le ((\log X) (\log \log X))^{-1}$. By \eqref{1st-asym} and \eqref{2nd-asym}, we then derive that
$$
 \left(\frac{1}{F-\eta } +o(1)\right) 
 \sum_{\substack{ \mathfrak{f} \in\mathcal{F}\\ \mathcal{Q}(\mathfrak{f};X) \in (\alpha,\beta) }}  \Phi\left(\frac{\norm(n)}{X}\right)
   =\sum_{\substack{ \mathfrak{f} \in\mathcal{F}\\ \mathcal{Q}(\mathfrak{f};X) \in (\alpha,\beta) }}     \sum_{\gamma_{\mathfrak{f} }} h\left( \frac{\gamma_{\mathfrak{f} } L}{2\pi}\right)  \Phi\left(\frac{\norm(n)}{X}\right)
$$
is greater or equal to
$$   
 0+ 
\sum_{\substack{ \mathfrak{f} \in\mathcal{F}  \backslash\mathcal{Z}\\ \mathcal{Q}(\mathfrak{f};X) \in (\alpha,\beta) }}  \Phi\left(\frac{\norm(n)}{X}\right) 
=\sum_{\substack{ \mathfrak{f} \in\mathcal{F} \\ \mathcal{Q}(\mathfrak{f};X) \in (\alpha,\beta) }}  \Phi\left(\frac{\norm(n)}{X}\right) 
-\sum_{\substack{ \mathfrak{f} \in\mathcal{F}\cap  \mathcal{Z}   \\ \mathcal{Q}(\mathfrak{f};X) \in (\alpha,\beta) }}  \Phi\left(\frac{\norm(n)}{X}\right) 
$$
as $\mathcal{F}$ is the disjoint union of $\mathcal{F}\cap  \mathcal{Z}$ and  $\mathcal{F}\backslash  \mathcal{Z} $. Hence, we deduce
\begin{align*}
 \sum_{\substack{ \mathfrak{f} \in\mathcal{F}\cap  \mathcal{Z}   \\ \mathcal{Q}(\mathfrak{f};X) \in (\alpha,\beta) }}  \Phi\left(\frac{\norm(n)}{X}\right) 
\ge  
\left(1 - \frac{1}{F-\eta } +o(1)\right)  \sum_{\substack{ \mathfrak{f} \in\mathcal{F} \\ \mathcal{Q}(\mathfrak{f};X) \in (\alpha,\beta) }}  \Phi\left(\frac{\norm(n)}{X}\right),
\end{align*}
as stated in the lemma.
\end{proof}

With these lemmas and propositions in hand, we are now ready to prove the main theorem.
\begin{proof}[Proof of Theorem \ref{main-thm}]
Recall the definition of $\mathcal{H}_X(\alpha,\beta)$ from Lemma \ref{dist-lemma}. For $\mathfrak{f}\in\mathcal{H}_X(\alpha,\beta)$, one has
$
\mathcal{Q}(\mathfrak{f};X) = \frac{\Re \mathcal{P}(\chi_\mathfrak{f}; x)   }{\sqrt{ \log \log X}}\in (\alpha,\beta) , 
$
while  $ L(s,\chi_{\mathfrak{f}})$ has no zeros $\frac{1}{2}+i\gamma_\mathfrak{f}$ with
$
 |\gamma_\mathfrak{f}|  \le( (\log X) (\log\log X))^{-1}.
$
Meanwhile, by Lemma \ref{RS-lemma2},  we can remove $\ll  X/\log \log\log X$ conductors $\mathfrak{f}$ from $\mathcal{H}_X(\alpha,\beta)$ so that the
contribution of zeros from $ L(s,\chi_{\mathfrak{f}})$ with $|\gamma_{\mathfrak{f}}| \ge ((\log X) (\log\log X))^{-1} $ to the last sum in \eqref{exp-logL} 
%with $\chi=\chi_{\mathfrak{f}}$, 
is bounded by $ (\log \log\log X)^3$. Therefore, there are at least 
$$
\left(\frac{2}{13}-\delta\right)
( \Psi(\alpha,\beta) +o(1) )
\# \{ \mathfrak{f} \in\mathcal{F} :  X < \norm(\mathfrak{f}) \le 2X \}
$$
conductors $\mathfrak{f}\in\mathcal{F}$ with $X< \norm(\mathfrak{f})\le 2X$ such that
$$
\frac{ \log | L(\oh , \chi_\mathfrak{f})|  }{\sqrt{ \log\log X}} = \frac{\Re \mathcal{P}(\chi_\mathfrak{f}; x)  }{\sqrt{ \log\log X}}
 +O\left(\frac{ (\log \log\log X)^3}{\sqrt{\log\log X}}\right)\in (\alpha,\beta),
$$
as desired. 
\end{proof}

\section{Proofs of key propositions}\label{section proofs of props}

In this section, we will prove Propositions \ref{Prop 1}-\ref{Prop 3}.

\subsection{Proof of Proposition \ref{Prop 1}}

%[Proof of Proposition \ref{Prop 1}]
The logarithmic derivative of the Hadamard factorisation of $L(s,\chi_\mathfrak{f})$ yields
\begin{equation}\label{hada}
\Re \frac{L'}{L}(\sigma +it,\chi_\mathfrak{f})
 =  -
\sum_{\rho =\frac{1}{2} +i\gamma_\mathfrak{f}}  \frac{\sigma- \frac{1}{2} }{(\sigma -\frac{1}{2})^2+ (t-\gamma_\mathfrak{f})^2} +O_{F}(\log \norm( \mathfrak{f})),
\end{equation}
for $\frac{1}{2}\le\sigma\le 1$ and $|t|\le 1$. Integrating $\sigma$ from $\frac{1}{2}$ to $\sigma_0 =\frac{1}{2}+\frac{1}{\log x}$ then yields
\begin{align*}
\log |L(\oh +it,\chi_\mathfrak{f})| -\log |L(\sigma_0 +it, \chi_\mathfrak{f})|= O\Big(\frac{\log \norm( \mathfrak{f})}{\log x}\Big)   
-\frac{1}{2}\sum_{\rho =\frac{1}{2} +i\gamma_\mathfrak{f}} \log \frac{ (\sigma_0 -\frac{1}{2})^2 +(t-\gamma_\mathfrak{f})^2  }{ (t-\gamma_\mathfrak{f})^2}.
\end{align*}

On the other hand, following  \cite[Eq. (11)]{RaSo} and \cite[Lemma 14]{So} (rooted in an identity of Selberg),  for $\sigma\ge\frac{1}{2}$, one has
\begin{align*}
 \frac{L'}{L}(\sigma,\chi_\mathfrak{f})
 &=\sum_{\n} \frac{\Lambda_F(\n) \chi_\mathfrak{f}(\n)}{\norm(\n)^{\sigma} } \frac{\log (x/\norm(\n))}{\log x}
-\frac{1}{\log x}\left(\frac{L'}{L}\right)'(\sigma,\chi_\mathfrak{f})\\
&+\frac{1}{\log x} \sum_{\rho =\frac{1}{2}+i\gamma_\mathfrak{f}} \frac{x^{\rho-\sigma}}{(\rho-\sigma)^2} 
+O\left( \frac{1}{x^{\sigma} \log x}\right)
\end{align*}
for $x\ge 3$. If $L(\sigma_0 , \chi_\mathfrak{f})\neq 0$, integrating both sides from $\sigma_0$ to $\infty$ and taking the real part, we obtain
\begin{align}\label{log-L-sigma0}
 \begin{split}
\log |L(\sigma_0 ,\chi_\mathfrak{f})|
& =
\Re \sum_{\n} \frac{\Lambda_F(\n) \chi_\mathfrak{f}(\n)}{\norm(\n)^{\sigma_0} \log \norm(\n)} \frac{\log (x/\norm(\n))}{\log x}
-\frac{1}{\log x}\Re \frac{L'}{L}(\sigma_0, \chi_\mathfrak{f})\\
&+\frac{1}{\log x}\sum_{\gamma_\mathfrak{f}} \Re \int_{\sigma_0}^\infty \frac{x^{\rho-\sigma}}{(\rho-\sigma)^2}d\sigma 
+O\left( \frac{1}{ \sqrt{x} (\log x)^2}\right).
  \end{split}
\end{align}
For the first sum in \eqref{log-L-sigma0}, the terms $\n= \p^k$ for $k\ge 3$ contribute at most $O(1)$, and the contribution of $\n= \p^2$ is
$$
\sum_{\p} \frac{ \chi_\mathfrak{f}^2(\p)}{2\norm(\p)^{2\sigma_0} } \frac{\log (x/\norm(\p^2))}{\log x}.
$$
This is $O(1)$ if $\chi_\mathfrak{f}$ is not quadratic, and it is
$
\frac{1}{2}\log\log x +O(1)
$
if $\chi_\mathfrak{f}^2 =1$.
For the second term on the right of \eqref{log-L-sigma0} under GRH, by \cite[Theorem 5.17]{IK}, one has 
$$
\frac{L'}{L}(\sigma_0,\chi_\mathfrak{f})
\ll \frac{1}{2\sigma_0 -1}  (\log \norm(\mathfrak{f}))^{2-2\sigma_0  }  + \log\log  \norm(\mathfrak{f})
\ll \log x +\log\log \norm(\mathfrak{f}).
$$
Finally, for $|\gamma_\mathfrak{f}\log x |\ge 1$, we know
$$
 \int_{\sigma_0}^{\infty} \frac{x^{\rho -\sigma}}{(\rho-\sigma)^2} d\sigma
\ll \frac{1}{\gamma_\mathfrak{f}^2}  \int_{\frac{1}{2}}^{\infty} x^{\frac{1}{2} -\sigma} d\sigma
\ll \frac{1}{\gamma_\mathfrak{f}^2 \log x}
\le  \log x.  
$$
Also, for  $|\gamma_\mathfrak{f}\log x |\le 1$, one has
$$
\Re \int_{\frac{1}{2} +\frac{1}{\log x}}^{\infty} \frac{x^{\rho -\sigma}}{(\rho -\sigma)^2} d\sigma 
\ll   \int_{\frac{1}{2} +\frac{1}{\log x}}^{\infty} \frac{ x^{\frac{1}{2} -\sigma} }{(\frac{1}{2} -\sigma)^2} d\sigma \ll \log x.
$$
Thus, the last sum over zero in \eqref{log-L-sigma0} is $\ll \sum_{\gamma_\mathfrak{f}} \log x.$ Gathering everything together, we complete the proof.

\subsection{Proof of Proposition \ref{Prop 2}}\label{pf-prop 2}

Similar to \cite[Proof of Lemma 5.2]{HW}, by the definition of the Dirichlet polynomials $\mathcal{P}(\chi_\mathfrak{f}; x)$ in \eqref{def Dirichlet polynomial}, the $k^\mathrm{th}$ power of $\mathcal{P}(\chi_\mathfrak{f}; x)$ can be written as
$$\mathcal{P}(\chi_\mathfrak{f}; x)^k
= \sum_{\n} \frac{ a_{k}(\n)\chi_\mathfrak{f}(\n)}{\norm(\n)^{\frac{1}{2}}} W(\n),
$$
where for $\n=\prod_{i=1}^{r} \p_{i}^{\alpha_{i}}$ with distinct $\p_i$, we define
$
W(\n) := \prod_{i=1}^{r} w(\p_{i})^{\alpha_{i}},
$
and
\begin{align*}
a_{k}(\n):=
\left\{
\begin{array}{lllll}\frac{k!}{\alpha_{1}!\cdots\alpha_{r}!}&\mbox{if $\n=\prod_{i=1}^{r} \p_{i}^{\alpha_{i}}$ for distinct $\p_i$ with $\norm(\p_{i})\leq x$, and $\sum_{i=1}^{r}\alpha_{i}=k$,}\\
0 &\mbox{otherwise}.
\end{array}\right.
\end{align*}
Therefore, 
\begin{align}\label{eq. Pkpj}
\begin{split}
    \sum_{\mathfrak{f}\in\mathcal{F}}  \mathcal{P}(\chi_\mathfrak{f}; x)^{k}\overline{\mathcal{P}(\chi_\mathfrak{f}; x)^{j}} \Phi\left(\frac{\norm(\mathfrak{f})}{X}\right)
&=\sum_{\mathfrak{f}\in\mathcal{F}}  \sum_{\n} \frac{ a_{k}(\n)a_{j}(\n)\chi_\mathfrak{f}\overline{\chi_\mathfrak{f}}(\n)}{\norm(\n)} W(\n)^2\Phi\left(\frac{\norm(\mathfrak{f})}{X}\right)\\
&+
O\bigg( 
\sum_{\mathfrak{f}\in\mathcal{F}}  \sum_{\n\neq\m} \frac{ a_{k}(\n)a_{j}(\m)\chi_\mathfrak{f}(\n)\overline{\chi_\mathfrak{f}}(\m)}{\norm(\n\m)^{\frac{1}{2}}} W(\n)W(\m)\Phi\left(\frac{\norm(\mathfrak{f})}{X}\bigg)
\right).
\end{split}
\end{align}

If $k\neq j$, $a_{k}(\n)a_{j}(\n)$ is zero by construction and $\n\ne\m$. So \eqref{eq. Pkpj} is bounded by
\begin{align*}
    \ll
    \sum_{\n}\frac{a_k(\n)W(\n)}{\norm(\n)^{\frac{1}{2}}}
    \sum_{\m} \frac{a_j(\m)W(\m)}{\norm(\m)^{\frac{1}{2}}}
    \sum_{\mathfrak{f} \in \mathcal{F}}\chi_\mathfrak{f}(\n\m^2) \Phi\left(\frac{\norm(\mathfrak{f})}{X}\right).
\end{align*}
The contribution of square-free $\n, \m$ dominates above, thus by Theorem \ref{theorem S_1} we have the upper bound
\begin{align*}
  \ll_{k, j}  \bigg(\sum_{\norm(\p)\le x}\frac{w(\p)}{\norm(\p)^{\frac{1}{2}}}\bigg)^k
    \bigg(\sum_{\norm(\p)\le x} \frac{w(\p)}{\norm(\p)^{\frac{1}{2}}}\bigg)^j
    \sum_{\mathfrak{f} \in \mathcal{F}}\chi_\mathfrak{f}(\ell) \Phi\left(\frac{\norm(\mathfrak{f})}{X}\right)
    \ll 
    x^{\frac{k+j}{2}-\epsilon} \sqrt{X}\norm(\ell)^{1/4}
    \ll
    x^{\frac{3k+4j}{4}-\epsilon}\sqrt{X},
\end{align*}
where $\ell=\n\m^2
%\p_1\cdots \p_k\p^2_{k+1}\cdots \p^2_{k+j}
$ is not a cube and $\norm(\ell)\le x^{k+2j}$.

When $k=j$, \eqref{eq. Pkpj} is 
\begin{align*}
\sum_{\n} \frac{ a_{k}(\n)^2W(\n)^2}{\norm(\n)} 
\sum_{\mathfrak{f}\in\mathcal{F}}
\chi_\mathfrak{f}(\n^3)
\Phi\left(\frac{\norm(\mathfrak{f})}{X}\right)
+
O\bigg( \sum_{\n\neq\m} \frac{ a_{k}(\n)a_{k}(\m)W(\n)W(\m)}{\norm(\n\m)^{\frac{1}{2}}} 
\sum_{\mathfrak{f}\in\mathcal{F}}
\chi_\mathfrak{f}(\n\m^2)
\Phi\left(\frac{\norm(\mathfrak{f})}{X}\right) \bigg).
\end{align*}
The non-square-free $\n$ in the first sum is of order at most $O(X(\log\log x)^{k-1})$, so the main term above is
\begin{align*}
&k!\sum_{\substack{ \norm(\p_1),\cdots, \norm(\p_k)\leq x\\\text{$\p_j$ distinct}}}\frac{W(\p_1\cdots \p_k)^2} {\norm(\p_1\cdots \p_k)} 
\sum_{\mathfrak{f}\in\mathcal{F}}
\chi_\mathfrak{f}(\p_1^3\cdots \p_k^3)
\Phi\left(\frac{\norm(\mathfrak{f})}{X}\right)\\
&= 
k!\Big(\sum_{\substack{\norm(\p)\leq x}}\frac{w(\p)^2}{\norm(\p)+1}\Big)^k
\frac{X\widehat{\Phi}(0)}{81\zeta_\mathcal{K}(2)} 
+
 O_{k}(X(\log\log x)^{k-1})
=
 \frac{k!X(\log \log x)^k\widehat{\Phi}(0)}{81\zeta_\mathcal{K}(2)}
 +
  O_{k}(X(\log\log x)^{k-1}),
\end{align*}
where the last equality follows from Mertens' estimate over $\mathcal{K}$.

Now, for any parameter $L$ such that $e^L\le X^{\frac{13}{11}}$, we also compute
\begin{align}\label{eq. Pkpj with h}
\begin{split}
&\sum_{\mathfrak{f}\in\mathcal{F}} \mathcal{P}(\chi_\mathfrak{f}; x)^{k}\overline{\mathcal{P}(\chi_\mathfrak{f}; x)^{j}} 
\sum_{\gamma_\mathfrak{f}} h\left(\frac{\gamma_\mathfrak{f} L}{2\pi}\right) \Phi\left(\frac{\norm(\mathfrak{f})}{X}\right)
\\&=
\sum_{\mathfrak{f}\in\mathcal{F}}  \sum_{\n} \frac{ a_{k}(\n)a_{j}(\n)\chi_\mathfrak{f}\overline{\chi_\mathfrak{f}}(\n)W(\n)^2}{\norm(\n)}  \sum_{\gamma_\mathfrak{f}} h\left(\frac{\gamma_\mathfrak{f} L}{2\pi}\right)
\Phi\left(\frac{\norm(\mathfrak{f})}{X}\right)
\\
& +O\bigg( 
\sum_{\mathfrak{f}\in\mathcal{F}}  \sum_{\n\neq\m} \frac{ a_{k}(\n)a_{j}(\m)\chi_\mathfrak{f}(\n)
\overline{\chi_\mathfrak{f}}(\m)W(\n)W(\m)}{\norm(\n\m)^{\frac{1}{2}}}  \sum_{\gamma_\mathfrak{f}} h\left(\frac{\gamma_\mathfrak{f} L}{2\pi}\right) \Phi\left(\frac{\norm(\mathfrak{f})}{X}\right)\bigg).
\end{split}
\end{align}
When $k\ne j$, the product $a_{k}(\n)a_{j}(\n)$ vanishes, so
\begin{align}\label{eq. prop 3 error with h}
\begin{split}
    &\sum_{\mathfrak{f}\in\mathcal{F}} \mathcal{P}(\chi_\mathfrak{f}; x)^{k}\overline{\mathcal{P}(\chi_\mathfrak{f}; x)^{j}} \sum_{\gamma_\mathfrak{f}} h\left(\frac{\gamma_\mathfrak{f} L}{2\pi}\right) \Phi\left(\frac{\norm(\mathfrak{f})}{X}\right)
    \\
    &\ll
    \sum_{\n}\frac{a_k(\n)W(\n)}{\norm(\n)^{\frac{1}{2}}}
    \sum_{\m \ne \n} \frac{a_j(\m)W(\m)}{\norm(\m)^{\frac{1}{2}}}
    \sum_{\mathfrak{f}\in \mathcal{F}} \sum_{\gamma_\mathfrak{f}} h\left(\frac{\gamma_\mathfrak{f} L}{2\pi}\right)
    \chi_\mathfrak{f}(\n\m^2)
    \Phi\left(\frac{\norm(\mathfrak{f})}{X}\right).
\end{split}
\end{align}
By Proposition \ref{Prop 3}, we consider two cases, where first we treat $\n\m^2$ not equal to a prime times a cube or a prime square times a cube.
In this case, the contribution is dominated by $\m, \n$ square-free, and thus we bound \eqref{eq. prop 3 error with h} by
\begin{align}\label{eq. prop 3 error with h not prime times cube}
    \ll
  x^{\frac{k+j}{2}-\epsilon}x^{\frac{14(k+2j)}{27}}e^{\frac{11L}{27}}X^{\frac{14}{27}+\epsilon}/{L}
    \ll
    e^{\frac{11L}{27}}X^{\frac{14}{27}+\epsilon}/{L}.
\end{align}
We also need to compute the contribution when $\n\m^2$ is a prime times a cube or a prime square times a cube, and it suffices to consider square-free $\n, \m$. Now, $\n\m^2=qa^3$ if $k=j+1$ and $\m=a, \n=q\m$.
%, and $\n\m^2=q^2a$ if $k=j-1$ and $\n=a, \m=q\n.$ 
Thus, \eqref{eq. prop 3 error with h} is bounded by
\begin{align*}
    &\ll
    \sum_{\norm(q)\le x } \frac{w(q)}{\sqrt{\norm(q)}}
    \sum_{\m}
    \frac{a_{j+1}(q\m)a_j(\m)W(\m)^2}{\norm(\m)}
    \sum_{\mathfrak{f}\in \mathcal{F}} \sum_{\gamma_\mathfrak{f}} h\left(\frac{\gamma_\mathfrak{f} L}{2\pi}\right)
    \chi_\mathfrak{f}(q\m^3)
    \Phi\left(\frac{\norm(\mathfrak{f})}{X}\right)
    \\
    &\ll
    \frac{{(j+1)!}
    X\left(\log \log x\right)^{j}}{L}
    \sum_{\norm(q)\le x } \frac{\log q}{{\norm(q)}}
    \ll
    \frac{{(j+1)!}
    X\log x \left(\log \log x\right)^{j}}{L}
    \ll
    {(j+1)!}
    X \left(\log \log x\right)^{j}.
\end{align*}
Since the case of $k=j-1$ and $\n=a, \m=q\n$ is symmetric, we conclude that \eqref{eq. prop 3 error with h} is bounded by 
\begin{align*}
    \ll{(\max\{k, j\})!}
    X \left(\log \log x\right)^{\min\{k,j\}}
    +
    e^{\frac{11L}{27}}X^{\frac{14}{27}+\epsilon}/{L}
\end{align*}
as desired.

When $k=j$, \eqref{eq. Pkpj with h} can be written as
\begin{align*}
%\sum_{\mathfrak{f}\in\mathcal{F}} \big|\mathcal{P}(\chi_\mathfrak{f}; x)\big|^{2k}\sum_{\gamma_{\mathfrak{f}}} &h\left( \frac{\gamma_{\mathfrak{f}} L}{2\pi}\right)
%\Phi\left(\frac{\norm(\mathfrak{f})}{X}\right)
&=
k!
\sum_{\substack{ \norm(\p_1),\cdots, \norm(\p_k)\leq x\\\text{$\p_j$ distinct}}}\frac{W(\p_1\cdots \p_k)^2} {\norm(\p_1\cdots \p_k)} \sum_{\mathfrak{f}\in\mathcal{F}} 
    \sum_{\gamma_\mathfrak{f}} 
     \chi_\mathfrak{f}(\p_1^3\cdots\p_k^3)h\left(\frac{\gamma_\mathfrak{f} L}{2\pi}\right) \Phi\left(\frac{\norm(\mathfrak{f})}{X}\right)
\\
& +O\bigg( 
\sum_{\mathfrak{f}\in\mathcal{F}}  \sum_{\n\neq\m} \frac{ a_{k}(\n)a_{k}(\m)W(\n)W(\m)}{\norm(\n\m)^{\frac{1}{2}}}  
\sum_{\gamma_\mathfrak{f}} \chi_\mathfrak{f}(\n\m^2)h\left(\frac{\gamma_\mathfrak{f} L}{2\pi}\right) \Phi\left(\frac{\norm(\mathfrak{f})}{X}\right)\bigg)
+
O_{k}(X(\log\log x)^{k-1+\varepsilon}),
\end{align*}
where the main term is again dominated by square-free $\n$, with the non-square-free contributing $O_{k}(X(\log\log x)^{k-1+\varepsilon}).$ 
Therefore the equation above is, for $\ell=\n^3$ in Proposition \ref{Prop 3},
\begin{align*}
    &=
    k!
    \bigg(\sum_{\substack{ \norm(\p)\leq x}}\frac{w(\p)^2} {\norm(\p)+1}\bigg)^k 
    \frac{X\widehat{\Phi}(0)}{81\zeta_\mathcal{K}(2)L}    
    \left(\hat{h}(0) \log X 
    + 
    O(1)\right)
    +
    O_k\bigg(\frac{x^{\frac{23k}{9}}e^{\frac{11L}{27}}X^{\frac{14}{27}+\epsilon}}{L}\bigg)
    +
    O_{k}(X(\log\log x)^{k-1+\varepsilon})
    \\
    &=
    \frac{k!X(\log\log x)^k\widehat{\Phi}(0)}{81\zeta_\mathcal{K}(2)L}  
    \left(\hat{h}(0) \log X 
    +     O(1)\right) 
    +    O_{k}(X(\log\log x)^{k-1+\varepsilon}).
\end{align*}
Since $\log \log x \sim  \log \log X,$
we derive the desired statements in the proposition.

%\textcolor{violet}
%{
%With $x^{k+j}\le X$ and $x=X^{\frac{1}{\log \log \log X}}$ as RS, it is $$j+k\le \log \log \log X.$$ Since we will modify to
%$$x=X^{\frac{13/22}{\log \log \log X}},$$ this gives
%$$x^{\frac{3k+4j}{2}}\le X\implies 3k+4j\le \frac{44}{13}\log \log \log X,$$
%which is about the same as the original.
%}

%\pj{ Mertens' estimate here is
%$$
%\sum_{\norm(\p)\leq z} \frac{\chi(\p)}{\norm(\p)} = \delta(\chi) \log\log z + O_{\chi}(1), 
%$$
%where $\delta(\chi)=1$ if $\chi$ is principal/trivial Hecke character; otherwise $\delta(\chi)=0$. This can be obtained by the standard estimate
%$$
%\sum_{\norm(\p)\leq z}  \chi(\p) 
%=\delta(\chi) Li(z) + O_{\chi}(\frac{z}{(\log z)^A}).
%$$
%and summation by parts. Nonetheless, as we assume GRH, we can deduce a strong estimate (and track the dependence on $\chi$ in Mertens' estimate via
%$$
%\sum_{\norm(\p)\leq z}  \chi(\p) 
%=\delta(\chi) Li(z) + O(z^{1/2} \log \norm(\cond(\chi)  )z ))
%$$
% if needed.
%}

\subsection{Proof of Proposition \ref{Prop 3}}\label{Proof of prop 3}
For brevity, we use the notation $|\cdot|$ to denote the ideal norm throughout this section. For example, the norm of the ideal generated by $\pi\in \mathbb{Z}[\omega]$ is denoted by $|\pi|=\pi\overline{\pi}$. Recall that $\zeta_\mathcal{K}(s)$ is the Dedekind zeta function of $\mathcal{K}=\mathbb{Q}(\omega).$
Recall the family of conductors $\mathcal{F}$ defined in \eqref{def of thin family} and the sum over the zeros in Proposition \ref{Prop 3}
\begin{align}
\label{def sum over the zeros}
    D^T(X; \ell, h, \Phi)=\sum_{\mathfrak{f}\in\mathcal{F}} \sum_{\gamma_{\mathfrak{f}}} h\left( \frac{\gamma_{\mathfrak{f}} L}{2\pi}\right)\chi_\mathfrak{f}(\ell) \Phi\left(\frac{|\mathfrak{f}|}{X}\right).
\end{align}
We want to estimate \eqref{def sum over the zeros}, where $L\ge1$ is a real number, $\ell\in \mathbb{Z}[\omega]$ is coprime to 3, and the functions $h$ and $\Phi$ are defined in Section \ref{background}. First, we apply the explicit formula (Lemma \ref{explicit formula}), then write 
\begin{align*}
    \sum_{\mathfrak{f}\in\mathcal{F}} \sum_{\gamma_{\mathfrak{f}}} h\left( \frac{\gamma_{\mathfrak{f}} L}{2\pi}\right)\chi_\mathfrak{f}(\ell) \Phi\left(\frac{|\mathfrak{f}|}{X}\right)
    =S_1-S_2,
\end{align*}
where
\begin{align*}
	S_1
	=
	\frac{1}{2\pi} \int_{-\infty}^\infty h\left(\frac{tL}{2\pi}\right)\sum_{\mathfrak{f}\in \mathcal{F}}\chi_\mathfrak{f}(\ell)\left(\log \frac{3|\mathfrak{f}|}{4\pi^2} + 2\frac{\Gamma^\prime}{\Gamma}\left(\frac{1}{2}+it\right)\right)\Phi\left(\frac{|\mathfrak{f}|}{X}\right)dt,
\end{align*}
and
\begin{align}\label{eq. S_2 sum}
	S_2
	=
	\frac{1}{L}\sum_{n \in \mathbb{Z}[\omega]} \frac{\Lambda_\mathcal{K}(n)}{\sqrt{|n|}}\hat{h}\left(\frac{\log |n|}{L}\right)\sum_{\mathfrak{f}\in \mathcal{F}}(\chi_\mathfrak{f}(\ell n)+\chi_\mathfrak{f}(\ell n^2)) \Phi\left(\frac{|\mathfrak{f}|}{X}\right).
\end{align}
The $S_1$ sum
%$\sum_{\mathfrak{f}\in \mathcal{F}}\chi_{\mathfrak{f}}(\ell)\Phi\left(\frac{|\mathfrak{f}|}{X}\right)$
is evaluated in the proof of Proposition \ref{Prop 3} by Theorem \ref{theorem S_1}, and the more difficult $S_2$ sum is handled in Theorems \ref{theorem S_2 chi} and \ref{theorem S_2}.

\subsubsection{Twisted zero-sum estimates}
We wrtie the $S_2$ sum defined in \eqref{eq. S_2 sum} as
$$S_2=S_2(\chi)+S_2(\overline{\chi}),$$
where
\begin{align*}
    S_2(\chi) = 
    \frac{1}{L}\sum_{n \in \mathbb{Z}[\omega]} \frac{\Lambda_\mathcal{K}(n)}{\sqrt{|n|}}\hat{h}\left(\frac{\log |n|}{L}\right)
    \sum_{\mathfrak{f}\in \mathcal{F}} \chi_\mathfrak{f}(\ell n) 
     \Phi\left(\frac{|\mathfrak{f}|}{X}\right),
\end{align*}
and
\begin{align*}
    S_2(\overline{\chi})
    =
    \frac{1}{L}\sum_{n \in \mathbb{Z}[\omega]} \frac{\Lambda_\mathcal{K}(n)}{\sqrt{|n|}}\hat{h}\left(\frac{\log |n|}{L}\right)
    \sum_{\mathfrak{f}\in \mathcal{F}} \chi_\mathfrak{f}(\ell n^2) 
    \Phi\left(\frac{|\mathfrak{f}|}{X}\right).
\end{align*}
Here, $L, \ell, h,$ and $\Phi$ are given in Section \ref{background}. We estimate $S_2(\chi)$ first in the following theorem and treat $S_2(\overline{\chi})$ in a similar manner in the proof of Theorem \ref{theorem S_2}.

\begin{theorem}\label{theorem S_2 chi}
Assume GRH, and let $e^{\frac{11L}{27}}|\ell|^{\frac{14}{27}}\le X^{\frac{13}{27}}$. If $\ell$ is a cube, then we have
    \begin{align*}
    S_2(\chi)
    =
    \frac{X\widehat{\Phi}(0)C_{3, h}\zeta^{-1}_\mathcal{K}(2)  \prod_{p\mid \ell}(1+\frac{1}{|p|})^{-1}
    }{|9|L}
    +O\bigg(\frac{X^{\frac{14}{27}+\epsilon}e^{\frac{11L}{27}}|\ell|^{\frac{14}{27}}}{L} \bigg),
    \end{align*}
    where $C_{3,h}$ is the constant defined in \eqref{eq. the c constant}.
    
    When $\ell=qa^3$ for a unique prime $q$, 
    \begin{align*}
    S_2(\chi)
    \ll
     \frac{X\widehat{\Phi}(0)\log|q|\zeta^{-1}_\mathcal{K}(2)  \prod_{p\mid q \ell }(1+\frac{1}{|p|})^{-1}}{|9|(|q|-|q|^{-1/2})L} 
    + 
    \frac{X^{\frac{14}{27}+\epsilon}e^{\frac{11L}{27}}|\ell|^{\frac{14}{27}}}{L},
    \end{align*}
    and when $\ell=q^2a^3$ for a unique prime $q$, 
    \begin{align*}
    S_2(\chi)
    \ll
     \frac{X\widehat{\Phi}(0)\log|q|\zeta^{-1}_\mathcal{K}(2)  \prod_{p\mid q \ell }(1+\frac{1}{|p|})^{-1}}{|9|(\sqrt{|q|}-|q|^{-1})L} 
    + 
    \frac{X^{\frac{14}{27}+\epsilon}e^{\frac{11L}{27}}|\ell|^{\frac{14}{27}}}{L}.
    \end{align*}
    When neither is the case, 
    $S_2(\chi)
    \ll
    X^{\frac{14}{27}+\epsilon}e^{\frac{11L}{27}}|\ell|^{\frac{14}{27}}/{L}.$
    \end{theorem}

\begin{proof}

By \eqref{detect-square}, writing $\mathfrak{f}=d^2f$ with $d^2f\equiv 1\pmod 9$ and switching the order of summation, we obtain 
\begin{align}\label{eq. S_2_chi}
	S_2(\chi)
    =
    \frac{1}{L}
    \sum_{\substack{d\in \mathbb{Z}[\omega]\\ |d|\le A}}\mu(d)
	\sum_{\substack{n\in \mathbb{Z}[\omega]}}\frac{\Lambda_\mathcal{K}(n)}{\sqrt{|n|}}
	\hat{h}\left(\frac{\log |n|}{L}\right)
	\chi_{d^2}(\ell n)
	\sum_{\substack{f\in \mathbb{Z}[\omega]\\f\equiv \overline{d^2}\mymod 9}} \chi_{\ell n}(f) \Phi\left(\frac{|fd^2|}{X}\right) + O\left(\frac{X\log X}{|9|LA}\right).
\end{align}
The error term comes from the contribution of $|d|>A$, since, under GRH,
\begin{align*}
	&
    \frac{1}{L}\sum_{\substack{d\in \mathbb{Z}[\omega]\\ |d|> A}}\mu(d)
	\sum_{\substack{n\in \mathbb{Z}[\omega]}}\frac{\Lambda_\mathcal{K}(n)}{\sqrt{|n|}}
	\hat{h}\left(\frac{\log |n|}{L}\right)
	\chi_{d^2}(\ell n)
	\sum_{\substack{f\in \mathbb{Z}[\omega]\\f\equiv \overline{d^2}\mymod 9}} \chi_{\ell n}(f) \Phi\left(\frac{|fd^2|}{X}\right)
    \\
    &\ll
    \frac{X}{|9|L}
    \sum_{\substack{d\in \mathbb{Z}[\omega]\\ |d|> A}}\frac{1}{|d|^2}
	\sum_{\substack{n\in \mathbb{Z}[\omega]\\|n|\le e^L}}\frac{\Lambda_\mathcal{K}(n)}{\sqrt{|n|}} \chi_{d^2}(n)
    \ll 
    \frac{X\log X}{|9|LA}.
\end{align*} 

Due to the character modulo $\ell n$, 
\begin{align*}
    \sum_{f\equiv \overline{d^2}\mymod 9} \chi_{\ell n}(f)\Phi(|fd^2|/X)
    =
    \sum_{r\mymod{\ell n}}\sum_{\substack{f\equiv \overline{d^2}\mymod 9\\f\equiv r\mymod{n\ell}}} \chi_{\ell n}(f)\Phi(|fd^2|/X),
\end{align*}
and we have
$$f\equiv r\mymod {\ell n},  \ f\equiv \overline{d^2} \mymod 9 \implies f\equiv r9\overline{9}+\overline{d^2}\ell n\overline{\ell n}\mymod {9\ell n} $$
by the Chinese Remainder Theorem. 
%Here, $$\ell n\overline{\ell n}\equiv 1\mymod 9, \text{ and } 9\overline{9}\equiv 1\mymod {\ell n}.$$
Applying the Poisson summation formula (Lemma \ref{lemma Poisson}) to the sum over $f$, we deduce
\begin{align*}
    &\sum_{f\equiv \overline{d^2}\mymod 9} \chi_{\ell n}(f)\Phi\left(\frac{|fd^2|}{X}\right)
    =
    \frac{X}{|9{\ell n}d^2|}\sum_{k\in\mathbb{Z}[\omega]} \widehat{\Phi}\bigg(\sqrt{\frac{X|k|}{|9{\ell n}d^2|}}\bigg)
    \sum_{r\mymod {\ell n}} \chi_{\ell n}(r)e\bigg(\frac{-(r9\overline{9}+\overline{d^2}n\ell\overline{n\ell})k}{{9\ell n}\sqrt{-3}}\bigg),
\end{align*}
and the first term in \eqref{eq. S_2_chi} becomes
\begin{align}\label{eq. S_2_chi_PS}
   \begin{split}
    &\frac{X}{|9|L}
    \sum_{\substack{d\in \mathbb{Z}[\omega]\\ |d|\le A}}\frac{\mu(d)\overline{\chi_d(\ell)}}{|d|^2} \\
    &\times
	\sum_{\substack{n\in \mathbb{Z}[\omega]}}\frac{\Lambda_\mathcal{K}(n)}{\sqrt{|n|}}
	\hat{h}\left(\frac{\log |n|}{L}\right)
	\chi_{d^2}( n)
    \sum_{k\in\mathbb{Z}[\omega]} \widehat{\Phi}\left(\sqrt{\frac{X|k|}{|9{\ell n}d^2|}}\right)
    \frac{1}{|\ell n|}\sum_{r\mymod {\ell n}} \chi_{\ell n}(r)e\bigg(\frac{-(r9\overline{9}+\overline{d^2}n\ell\overline{n\ell})k}{{9\ell n}\sqrt{-3}}\bigg).
    \end{split}
\end{align}

We first compute the main term from $k=0$, and postpone the treatment of the error from $k\ne 0$. 
When $k=0$, the innermost character sum gives $\phi_\mathcal{K}(\ell n)$ when $\chi_{\ell n}$ is the principal character, and returns 0 otherwise. This implies that $\ell n$ must be a cube. Since $n$ must be a prime power, we are left to consider the cases that, for some $a \in\mathbb{Z}[\omega]$, 
$$\ell=a^3, n=p^{3\alpha} \text{ for any prime } p ; \ \  \ell=p_1a^3, n=p_1^{3\alpha+2}; \ \  \text{ or } \ \  \ell=p_2^2a^3, n=p_2^{3\alpha+1},$$
where in the latter two cases, $n$ is a power of a unique prime $p_1, p_2.$

Set $k=0$ in \eqref{eq. S_2_chi_PS} and consider $n$ being one of the cases described above, we have
\begin{align}
\label{eq. S_2 main term}
\begin{split}
	&\frac{X\widehat{\Phi}(0)}{|9|L}
    \sum_{\substack{(d, n\ell)=1\\ |d|\le A}}\frac{\mu(d)}{|d|^2}
	\sum_{\substack{n\in \mathbb{Z}[\omega]}}
    \frac{\phi_\mathcal{K}(\ell n)}{|\ell n|}   
    \frac{\Lambda_\mathcal{K}(n)}{\sqrt{|n|}}
	\hat{h}\left(\frac{\log |n|}{L}\right)
    \\
    &=
    \frac{X\widehat{\Phi}(0)}{|9|L}
	\sum_{\substack{n\in \mathbb{Z}[\omega]}}
 \frac{\Lambda_\mathcal{K}(n)}{\sqrt{|n|}}  
 \frac{\phi_\mathcal{K}(\ell n)}{|\ell n|} 
 \sum_{\substack{(d, n\ell)=1\\ |d|\le A}}\frac{\mu(d)}{|d|^2}
	\hat{h}\left(\frac{\log |n|}{L}\right).
    \end{split}
\end{align}
As $$
\sum_{\substack{(d, n\ell)=1\\ |d|\le A}}\frac{\mu(d)}{|d|^2}
= 
\sum_{\substack{d\in \mathbb{Z}[\omega]\\(d, n\ell)=1}}\frac{\mu(d)}{|d|^2}
+O\left(\frac{1}{A}\right)
=
\prod_{p \nmid n\ell} \left(1-\frac{1 }{|p|^2}\right) +O\left(\frac{1}{A}\right) ,
$$
and
$$\frac{\phi_\mathcal{K}(n\ell)}{|n\ell|}=\prod_{p\mid n\ell} \left(1-\frac{1}{|p|}\right),$$
we have
\begin{align}
\label{eq. Euler product S_2}
\frac{\phi_\mathcal{K}(n\ell)}{|n\ell|}
\sum_{\substack{(d, n\ell)=1\\ |d|\le A}}\frac{\mu(d)}{|d|^2}
&=
\prod_{p\mid n\ell} \left(1-\frac{1}{|p|}\right) \bigg(\prod_{p \nmid n\ell} \left(1-\frac{1 }{|p|^2}\right) +O\left(\frac{1}{A}\right)\bigg) \notag
\\
&=
\frac{\prod_p(1-\frac{1}{|p|^2})}{\prod_{p\mid n\ell}(1+\frac{1}{|p|})} + O\left(\frac{1}{A}\right)
=
\zeta^{-1}_\mathcal{K}(2)  \prod_{p\mid n\ell}\left(1+\frac{1}{|p|}\right)^{-1} +O\left(\frac{1}{A}\right).
\end{align}

Substituting \eqref{eq. Euler product S_2} into \eqref{eq. S_2 main term} we have
\begin{align}\label{eq. S_2 after Euler product}
\begin{split}
    &\frac{X\widehat{\Phi}(0)}{|9|L}
	\sum_{\substack{n\in \mathbb{Z}[\omega]}}
 \frac{\Lambda_\mathcal{K}(n)}{\sqrt{|n|}}  
 \frac{\phi_\mathcal{K}(\ell n)}{|\ell n|} 
\bigg(\prod_{p \nmid n\ell} \left(1-\frac{1 }{|p|^2}\right) +O\left(\frac{1}{A}\right)\bigg)
	\hat{h}\left(\frac{\log |n|}{L}\right) 
    \\
    &=
    \frac{X\widehat{\Phi}(0)}{|9|L}
	\sum_{\substack{n\in \mathbb{Z}[\omega]}}
 \frac{\Lambda_\mathcal{K}(n)}{\sqrt{|n|}}  
 \bigg(\zeta^{-1}_\mathcal{K}(2)  \prod_{p\mid n\ell}\left(1+\frac{1}{|p|}\right)^{-1} +O\left(\frac{1}{A}\right)\bigg)
	\hat{h}\left(\frac{\log |n|}{L}\right).
    \end{split}
\end{align}
In the first case, $\ell=a^3$ for some $a\in\mathbb{Z}[\omega]$ and we must have $n=q^3$ for some prime power $q$, therefore
\begin{align}
\label{eq. main term 1}
   \begin{split}
    &\frac{X\widehat{\Phi}(0)}{|9|L}
	\sum_{\substack{q\in \mathbb{Z}[\omega] }}
 \frac{\Lambda_\mathcal{K}(q)}{\sqrt{|q|^3}}  
\bigg(\zeta^{-1}_\mathcal{K}(2)  \prod_{p\mid q\ell}\left(1+\frac{1}{|p|}\right)^{-1} +O\left(\frac{1}{A}\right)\bigg)
	\hat{h}\left(\frac{3\log |q|}{L}\right)
    \\
    &= \frac{X\widehat{\Phi}(0)}{|9|L}\zeta^{-1}_\mathcal{K}(2)  \prod_{p\mid \ell}\left(1+\frac{1}{|p|}\right)^{-1}
	\sum_{\substack{q\in \mathbb{Z}[\omega] }}
 \frac{\Lambda_\mathcal{K}(q)}{\sqrt{|q|^3}}  
\bigg( \prod_{p\mid q}\left(1+\frac{1}{|p|}\right)^{-1} +O\left(\frac{1}{A}\right)\bigg)
	\hat{h}\left(\frac{3\log |q|}{L}\right)
    \\
    &=
    \frac{X\widehat{\Phi}(0)C_{3, h}}{|9|L}
    \zeta^{-1}_\mathcal{K}(2)  \prod_{p\mid \ell}\left(1+\frac{1}{|p|}\right)^{-1}
	+ O\left(\frac{X}{AL}\right),
    \end{split}
\end{align}
where $C_{3,h}$ is the constant defined as in \eqref{eq. the c constant}.

When $\ell=p_1a^3$ for a unique prime $p_1$, we have $n=p_1^{3\alpha+2}$, and \eqref{eq. S_2 after Euler product} becomes
\begin{align}
\label{eq. main term 2}
 \begin{split}
    &\frac{X\widehat{\Phi}(0)\log|p_1|\left(\zeta^{-1}_\mathcal{K}(2)  \prod_{p\mid p_1\ell}(1+\frac{1}{|p|})^{-1}+O\left(\frac{1}{A}\right)\right)}{|9 p_1|L}
    \sum_{\alpha=0}^\infty
    \frac{\hat{h}\left(\frac{(3\alpha+2)\log |p_1|}{L}\right)}{|p_1|^{3\alpha/2}} 
    \\
    &\ll
    \frac{X\widehat{\Phi}(0)\log|p_1|\zeta^{-1}_\mathcal{K}(2)  \prod_{p\mid p_1\ell}(1+\frac{1}{|p|})^{-1}}{|9 p_1|L}
    \sum_{\alpha=0}^{\frac{L}{3\log|p_1|}}
    \frac{1}{|p_1|^{3\alpha/2}}
    \ll
 \frac{X\widehat{\Phi}(0)\log|p_1|\zeta^{-1}_\mathcal{K}(2)  \prod_{p\mid p_1 \ell }(1+\frac{1}{|p|})^{-1}}{|9|(|p_1|-|p_1|^{-1/2})L}.
  \end{split}
\end{align}
Similarly, when $\ell=p_2^2a^3$ for a unique prime $p_2$, we have $n=p_2^{3\alpha+1}$, and
\begin{align}
\label{eq. main term 3}
 \begin{split}
    &\frac{X\widehat{\Phi}(0)\log|p_2|\left(\zeta^{-1}_\mathcal{K}(2)  \prod_{p\mid p_2\ell}(1+\frac{1}{|p|})^{-1}+O\left(\frac{1}{A}\right)\right)}{|9|\sqrt{|p_2|} L}
    \sum_{\alpha=0}^\infty
    \frac{\hat{h}\left(\frac{(3\alpha+1)\log |p_2|}{L}\right)}{|p_2|^{3\alpha/2}} 
    \\
    &\ll
    \frac{X\widehat{\Phi}(0)\log|p_2|\zeta^{-1}_\mathcal{K}(2)  \prod_{p\mid p_2\ell}(1+\frac{1}{|p|})^{-1}}{|9|\sqrt{|p_2|}L}
    \sum_{\alpha=0}^{\frac{L}{3\log|p_2|}}
    \frac{1}{|p_2|^{3\alpha/2}}
    \ll
 \frac{X\widehat{\Phi}(0)\log|p_2|\zeta^{-1}_\mathcal{K}(2)  \prod_{p\mid p_2 \ell }(1+\frac{1}{|p|})^{-1}}{|9|(\sqrt{|p_2|}-|p_2|^{-1})L}.
  \end{split}
\end{align}
%%%%%%%%%%%%%%%%%%%%
%the old computation without powers is
%\begin{align*}
%    \frac{X\widehat{\Phi}(0)}{|9|}
%	\sum_{\substack{n\in \mathbb{Z}[\omega]}}
% \frac{\Lambda_\mathcal{K}(n)}{\sqrt{|n|}}  
% \left(\zeta^{-1}_\mathcal{K}(2)  \prod_{p\mid n\ell}\left(1+\frac{1}{|p|}\right)^{-1} +O\left(\frac{1}{A}\right)\right)
%	\hat{h}\left(\frac{\log |n|}{L}\right)
%    \\
%    \ll
% \frac{X\widehat{\Phi}(0)\log|q|}{|9q^{1/2}|}  
% \left(\zeta^{-1}_\mathcal{K}(2)  \prod_{p\mid q\ell}\left(1+\frac{1}{|p|}\right)^{-1} \right).
%\end{align*}
%%%%%%%%%%%%%%%%%%%%%%

Now, we turn to the more difficult error contribution from $k\ne 0$. First, we observe that the contribution of $n=p^\alpha, \alpha\ge 3$ in \eqref{eq. S_2_chi_PS} is bounded by \ref{eq. final S_2 error bound}. 
When $n=p^\alpha,$ we have
\begin{align*}
    &\ll \frac{X}{|9|L}
    \sum_{\substack{d\in \mathbb{Z}[\omega]\\ |d|\le A}}\frac{1}{|d|^2}
	\sum_{\substack{p\in \mathbb{Z}[\omega]\\|p|\le e^{L/\alpha}}}\frac{\log |p|}{\sqrt{|p|^\alpha}}
    \frac{|9\ell d^2 p^\alpha|}{X}
    \ll
    \frac{Ae^{(\frac{1}{2}+\frac{1}{\alpha}+\epsilon)L}|\ell|}{L},
\end{align*}
so the upper bound in \eqref{eq. final S_2 error bound} suffices.
Thus, we are left with sums over $n=p, p^2.$ 
The exceptional cases of $\ell n $ being a cube can only occur when $\ell$ is a prime times a cube or a prime square times a cube and $n$ is the unique complementary prime. They contribute to $S_2(\chi)$ at most $O(Ae^{L/2+\epsilon}|\ell|/L)$.
We bound the rest via upper bounds of the prime sum over cubic Gauss sums. 

We write \eqref{eq. S_2_chi_PS}  as
\begin{align}\label{eq. S_2_chi_error1}
 \begin{split}
    &\frac{X}{|9|L}
    \sum_{\substack{d\in \mathbb{Z}[\omega]\\ |d|\le A}}\frac{\mu(d)\overline{\chi_d(\ell)}}{|d|^2} \\
    &\times
    \sum_{k\in\mathbb{Z}[\omega]}
    \sum_{\substack{n\in \mathbb{Z}[\omega]}}\chi_{d^2}( n)\frac{\Lambda_\mathcal{K}(n)}{|\ell|\sqrt{|n|^3}}
	\hat{h}\left(\frac{\log |n|}{L}\right)
	\widehat{\Phi}\left(\sqrt{\frac{X|k|}{|9{\ell n}d^2|}}\right)
    e\left(\frac{-\overline{d^2n\ell}k}{9\sqrt{-3}}\right)
    \chi_{\ell n }(-9\sqrt{-3})g(k, \ell n),
    \end{split}
\end{align}
where for the innermost sum, letting $\frac{-r\overline{9}}{\sqrt{-3}}\equiv s\mymod{\ell n}$,
\begin{align}
    \sum_{r\mymod {\ell n}} \chi_{\ell n}(r) 
    e\left(\frac{-r\overline{9}k}{{\ell n}\sqrt{-3}}\right)
    =
    \sum_{s\mymod{n\ell}}\chi_{\ell n}(-9\sqrt{-3}s)e\left(\frac{sk}{\ell n}\right) 
    =\chi_{\ell n }(-9\sqrt{-3})g(k, \ell n).
    \end{align}
If $(n, \ell)\ne 1$, the sum over $n=p, p^2$ is bounded by
\begin{align*}
    \sum_{\substack{n\in \mathbb{Z}[\omega]\\(n, \ell)\ne 1}}\frac{\Lambda_\mathcal{K}(n)}{\sqrt{|n|}}
	\hat{h}\left(\frac{\log |n|}{L}\right)
	\widehat{\Phi}\left(\sqrt{\frac{X|k|}{|9{\ell n}d^2|}}\right)
    e\left(\frac{-\overline{d^2n\ell}k}{9\sqrt{-3}}\right)
    \frac{g(k, \ell n)}{|\ell n|}\chi_{d^2}( n)\chi_{\ell n }(-9\sqrt{-3})
   \ll
    \sum_{p\mid \ell}\frac{\log |p|}{|p|}
    \ll
    |\ell|^\epsilon,
\end{align*}
and \eqref{eq. final S_2 error bound} suffices.
Thus, we are left to consider those $n$ such that $(n, \ell)=1.$ 
By Lemma \ref{lemma Gauss sum=0} for $a=0$, we have for $j\in \{1, 2\}$,
\begin{align}
\label{eq. Gauss sum splits}
    \frac{g(k, \ell p^j)}{|\ell|\sqrt{|p|^{3j}}}
    =
    \frac{\chi_{p^{j}}(\ell)\chi_{\ell}(p^{j})g(k, p^{j})g(k, \ell)}{|\ell||p|^{3j/2}}
    =
    \frac{g(k, \ell)}{|\ell||p|^{1+j/2}}
    \times
    \begin{cases}
        \overline{\chi_p(\ell)}g(k, p)   &\text{ if }    j=1,\\
        \chi_p(\ell)\overline{g\left(k^\prime, p\right)}   &\text{ if }    j=2,
    \end{cases}
\end{align}
where $p^{j-1} \mid \mid k$, and $k^\prime=k/p^{j-1},$ otherwise the sum vanishes.

We illustrate the prime case $(j=1)$ below in detail, while in the prime square case with $j=2$, 
$$\overline{\chi_p(d^2)} \overline{g(k^\prime, p)}
=\chi_p(d)\overline{g(k^\prime, p)}
=\overline{g(dk^\prime, p)},$$ 
which can be bounded through the same procedure as in the prime case. Using \eqref{eq. Gauss sum splits} when $j=1$,
the sum over $n$ above is
\begin{align}
\label{eq. partial S_2 error}
\begin{split}
    \frac{\chi_\ell(-9\sqrt{-3})g(k, \ell)}{|\ell|}
    \sum_{\substack{n\in \mathbb{Z}[\omega]}} \overline{\chi_{d}( n)}\frac{\Lambda_\mathcal{K}(n)}{\sqrt{|n|^3}}
	\hat{h}\left(\frac{\log |n|}{L}\right)
	\widehat{\Phi}\left(\sqrt{\frac{X|k|}{|9{\ell n}d^2|}}\right)
    e\left(\frac{-\overline{d^2n\ell}k}{9\sqrt{-3}}\right)
    \chi_{n}(-9\sqrt{-3})
    g(k\ell, n) 
    \\
    =
    \frac{\chi_\ell(-9\sqrt{-3})g(k, \ell)}{|\ell|}
    \sum_{\substack{n\in \mathbb{Z}[\omega]\\}} \frac{\Lambda_\mathcal{K}(n)}{\sqrt{|n|^3}}
	\hat{h}\left(\frac{\log |n|}{L}\right)
	\widehat{\Phi}\left(\sqrt{\frac{X|k|}{|9{\ell n}d^2|}}\right)
    e\left(\frac{-\overline{d^2n\ell}k}{9\sqrt{-3}}\right)
    \chi_{n}(-9\sqrt{-3})
    g(dk\ell, n).
    \end{split}
\end{align}
The bound the sum over cubic Gauss sums, we have from \cite[Theorem 4.4]{DG},
\begin{align*}
	H(r, Z, \lambda)
    &=
    \sum_{\substack{(c, r)=1\\
%    c\equiv 1\mymod 3\\
    |c|\le Z}}\frac{\Lambda(c)}{\sqrt{|c|}}g(r, c)\lambda(c)
    \\
	&\ll
	Z^\epsilon|r|^\epsilon\left(Z^{5/6}|r|^{1/12}+Z^{4/5}|r|^{1/10} +Z^{2/3}|r|^{1/6} + Z^{1/2}|r|^{1/4} \right)
    =Z^\epsilon|r|^\epsilon 
    \sum_{j=1}^4 Z^{\alpha_j}|r|^{\beta_j},
\end{align*}
where we denote for each integer $1\le j\le 4$, $(\alpha_j, \beta_j)$ the exponents on $Z$ and $|r|$ respectively.
%so the sum over $|n|\le e^L$ is expected to be bounded by
%\begin{align*}
%	e^{\epsilon L}|kd|^\epsilon(e^{2L/3}|kd|^{1/6} + e^{L/2}|kd|^{1/4} + e^{5L/6}|kd|^{1/12}+e^{4L/5}|kd|^{1/10})=Z^*.
%\end{align*}
By partial summation, we deduce
\begin{align*}
	&\sum_{\substack{n\in \mathbb{Z}[\omega]
%    \\n\equiv 1\mymod 3
    \\|n|\le Y} }
    \frac{\Lambda_\mathcal{K}(n)}{\sqrt{|n|^3}}g(dk\ell, n)\chi_n(-9\sqrt{-3})
	\hat{h}\left(\frac{\log |n|}{L}\right) 
	\widehat{\Phi}\left(\sqrt{\frac{X|k|}{|9d^2\ell n|}}\right)\\
	&=
	\int_3^Y 
	\hat{h}\left(\frac{\log Z}{L}\right) 
	\widehat{\Phi}\left(\sqrt{\frac{X|k|}{|9d^2\ell|Z}}\right)
	Z^{-1}
	\mathrm{d}\bigg(\sum_{{|n|\le Z}}\frac{\Lambda_\mathcal{K}(n)}{\sqrt{|n|}}g(dk\ell, n)\chi_n(-9\sqrt{-3})\bigg)
	\\
	&=
	\frac{\hat{h}(\frac{\log Y}{L}) 
		\widehat{\Phi}\left(\sqrt{\frac{X|k|}{|9d^2\ell Y|}}\right)H(dk\ell, Y, \chi_n)}{Y}
		-
		\int_3^Y H(dk\ell, Z, \chi_n)
		\frac{d}{dZ}
		\bigg( \frac{\hat{h}(\frac{\log Z}{L})}{Z} 
		\widehat{\Phi}\left(\sqrt{\frac{X|k|}{|9d^2\ell|Z}}\right)
		\bigg)
		\mathrm{d}Z.
\end{align*}
Recall that
$\hat{h}(Z)= \max \{1-\frac{\log Z}{L}, 0\},$ $\widehat{\Phi}(x)\ll_K x^{-K}$, and $\widehat{\Phi}^\prime (x) \ll_K x^{-K-1}$ for any $K$.
Since
$$\frac{\hat{h}(\frac{\log Z}{L})}{Z} \le \frac{1}{Z}\quad \text{and} \quad \widehat{\Phi}\left(\sqrt{\frac{X|k|}{|9d^2\ell|Z}}\right)\ll \left(\frac{Z|d^2\ell|}{X|k|}\right)^{K/2} \text{ for any } K,$$
we have 
%\pj{the first one below is $\frac{d}{dZ}?$}\textcolor{violet}{Yes!}
$$\frac{d}{dZ} \frac{\hat{h}(\frac{\log Z}{L})}{Z}  \ll \frac{1}{Z^2}  \quad \text{and} \quad  \widehat{\Phi}^\prime\left(\sqrt{\frac{X|k|}{|9d^2\ell|Z}}\right) \ll  Z^{-1}\left(\frac{Z|d^2\ell|}{X|k|}\right)^{K/2}.$$
Thus, we obtain
\begin{align*}
    \sum_{\substack{n\in \mathbb{Z}[\omega]\\
%    n\equiv 1\mymod 3\\
    |n|\le Y} }&\frac{\Lambda_\mathcal{K}(n)}{\sqrt{|n|^3}}g(dk\ell, n)\chi_n(-9\sqrt{-3})
	\hat{h}\left(\frac{\log |n|}{L}\right) 
	\widehat{\Phi}\left(\sqrt{\frac{X|k|}{|9d^2\ell n|}}\right)
    \\
    &\ll 
    \sum_{j=1}^4  |dk\ell|^{\beta_j} \min\bigg\{Y^{\alpha_j-1}, Y^{\alpha_j + K/2-1} \left(\frac{X|k|}{|d^2\ell|}\right)^{-K/2}\bigg\}.
\end{align*}

Now, \eqref{eq. S_2_chi_PS} can be bounded as
\begin{align}
\label{eq. final S_2 error bound}
 \begin{split}
	&\frac{X}{|9|L}
	\sum_{\substack{d\in \mathbb{Z}[\omega]\\|d|\le A}}\frac{\mu(d)\overline{\chi_d(\ell)}}{|d^2|}
	\sum_{k \in \mathbb{Z}[\omega]}
    \frac{\chi_\ell(-9\sqrt{-3}){g(k, \ell)}}{|\ell|}
    \\
    &\times
    \sum_{\substack{n\in \mathbb{Z}[\omega]
%    \\n\equiv 1\mymod 3
    \\|n|\le Y}}\frac{\Lambda_\mathcal{K}(n)}{\sqrt{|n|^3}}g(dk\ell, n)\chi_n(-9\sqrt{-3})
	\hat{h}\left(\frac{\log |n|}{L}\right) 
	\widehat{\Phi}\left(\sqrt{\frac{X|k|}{|9d^2\ell n|}}\right)e\left(\frac{-\overline{d^2n\ell}k}{9\sqrt{-3}}\right) 
    \\ 
    &\ll 
    \frac{X}{|9|L}\sum_{j=1}^4 |\ell|^{\beta_j}
    \sum_{|d|\le A}|d|^{\beta_j-2}
    \bigg(
    Y^{\alpha_j-1}\sum_{|k|\le \frac{Y|d^2\ell|}{X}}|k|^{\beta_j}
    +
    Y^{\alpha_j + K/2-1} \left(\frac{X}{|d^2\ell|}\right)^{-K/2}\sum_{|k|> \frac{Y|d^2\ell|}{X}} |k|^{\beta_j-K/2}
    \bigg)
    \\
    &\ll  
    \frac{X}{|9|L}\sum_{j=1}^4 |\ell|^{\beta_j}
    \bigg( 
    Y^{\alpha_j-1}\sum_{|d|\le A} |d|^{\beta_j-2} \left(\frac{Y|d^2\ell|}{X}\right)^{\beta_j+1} 
    +
    \left(\frac{|\ell|}{X}\right)^{K/2}Y^{\alpha_j + K/2 -1}\sum_{|d|\le A} |d|^{\beta_j-2+K}
    \left(\frac{Y|d^2\ell|}{X}\right)^{\beta_j+1-K/2} 
    \bigg) 
    \\ 
    &\ll
    \frac{1}{L}\sum_{j=1}^4{Y^{\alpha_j + \beta_j} |\ell|^{2\beta_j +1}\frac{A^{3\beta_j+1}}{X^{\beta_j}} },
 \end{split}
\end{align}
for all $K\ge 3$. 

We list the following values for convenience.
\begin{align*}
    \alpha_j+\beta_j=
    \begin{cases}
        \frac{11}{12}, \text{ for } j=1,\\
        \frac{9}{10}, \text{ for } j=2,\\
        \frac{5}{6}, \text{ for } j=3,\\
        \frac{3}{4}, \text{ for } j=4,\\
    \end{cases} \ \ \ \ \ 
    2\beta_j+1=
    \begin{cases}
        \frac{7}{6}, \text{ for } j=1,\\
        \frac{6}{5}, \text{ for } j=2,\\
        \frac{4}{3}, \text{ for } j=3,\\
        \frac{3}{2}, \text{ for } j=4,\\
    \end{cases} 
    \ \ \ \ \ 
    3\beta_j+1=
    \begin{cases}
        \frac{15}{12}, \text{ for } j=1,\\
        \frac{13}{10}, \text{ for } j=2,\\
        \frac{9}{6}, \text{ for } j=3,\\
        \frac{7}{4}, \text{ for } j=4.\\
    \end{cases} 
\end{align*}
To balance the errors, we pick
$$A=\frac{X^{13/27}(\log X)^{12/27}}{e^{11L/27}|\ell|^{14/27}}$$
so that
$$\frac{X\log X}{|9|LA}
+
\frac{1}{L}\sum_{j=1}^4{Y^{\alpha_j + \beta_j} |\ell|^{2\beta_j +1}\frac{A^{3\beta_j+1}}{X^{\beta_j}} }
\ll
\frac{X^{\frac{14}{27}+\epsilon}e^{\frac{11L}{27}}|\ell|^{\frac{14}{27}}}{L}
=o\left(X\right)
$$
for $e^{11L/27}|\ell|^{14/27}\le X^{13/27}.$ Finally, putting \eqref{eq. main term 1}, \eqref{eq. main term 2}, \eqref{eq. main term 3}, and \eqref{eq. final S_2 error bound} together gives the desired statement in the theorem.
\end{proof}

\subsubsection{$S_2$ sum estimate}

In the following theorem, we first estimate $S_2(\overline{\chi})$ in a similar way to $S_2(\chi)$ (with differences in the twisted Gauss sum analysis, but resulting in estimates analogous to those of Theorem \ref{theorem S_2 chi}), then combine the results for $S_2(\chi)$ to obtain estimates for 
$
S_2=S_2(\chi)+S_2(\overline{\chi}).
$

\begin{theorem}
\label{theorem S_2}
Assume GRH, and suppose $e^{\frac{11L}{27}}|\ell|^{\frac{14}{27}}\le X^{\frac{13}{27}}$.
%(Note that $L\le \frac{13}{11}\log X$ when $|\ell|=1$.) 
Let $C_{3,h}$ be defined as in \eqref{eq. the c constant}.
%\begin{align*}
%    S_2(\overline{\chi})
%    =
%    \frac{1}{L}\sum_{n \in \mathbb{Z}[\omega]} \frac{\Lambda_\mathcal{K}(n)}{\sqrt{|n|}}\hat{h}\left(\frac{\log |n|}{L}\right)
%    \sum_{\mathfrak{f}\in \mathcal{F}} \chi_\mathfrak{f}(\ell n^2) 
%    \Phi\left(\frac{|\mathfrak{f}|}{X}\right),
%\end{align*}
If $\ell$ is a cube, then 
    \begin{align*}
    S_2
    =
    \frac{2X\widehat{\Phi}(0)C_{3, h}\zeta^{-1}_\mathcal{K}(2)  \prod_{p\mid \ell}(1+\frac{1}{|p|})^{-1}
    }{|9|L}
    +O\bigg(\frac{X^{\frac{14}{27}+\epsilon}e^{\frac{11L}{27}}|\ell|^{\frac{14}{27}}}{L} \bigg).
    \end{align*}

    When $\ell=qa^3$ or $\ell=q^2a^3$ for a unique prime $q$, 
    \begin{align*}
    S_2
    \ll
     \frac{X\widehat{\Phi}(0)\log|q|(1+\sqrt{|q|})\zeta^{-1}_\mathcal{K}(2)  \prod_{p\mid q \ell }(1+\frac{1}{|p|})^{-1}}{|9|(|q|-|q|^{-1/2})L} 
    + 
    \frac{X^{\frac{14}{27}+\epsilon}e^{\frac{11L}{27}}|\ell|^{\frac{14}{27}}}{L}.
    \end{align*}
    When neither is the case, 
    $$
    S_2
    \ll
    \frac{X^{\frac{14}{27}+\epsilon}e^{\frac{11L}{27}}|\ell|^{\frac{14}{27}}}{L}.$$
\end{theorem}
\begin{proof}
We follow similar steps as in Theorem \ref{theorem S_2 chi} to treat $S_2(\overline{\chi}).$ Selecting square-free conductors $\mathfrak{f}$, we obtain 
\begin{align}
\label{eq. S_2_chi_bar}
\begin{split}
    S_2(\overline{\chi})
    =
    \frac{1}{L}
    \sum_{\substack{d\in \mathbb{Z}[\omega]\\ |d|\le A}}\mu(d)
	\sum_{\substack{n\in \mathbb{Z}[\omega]}}\frac{\Lambda_\mathcal{K}(n)}{\sqrt{|n|}}
	\hat{h}\left(\frac{\log |n|}{L}\right)
	\chi_{d^2}(\ell n^2)
	\sum_{\substack{f\in \mathbb{Z}[\omega]\\f\equiv \overline{d^2}\mymod 9}} \chi_{\ell n^2}(f) \Phi\left(\frac{|fd^2|}{X}\right) + O\left(\frac{X\log X}{|9|LA}\right),
    \end{split}
\end{align}
where the same trivial bound for terms with $|d|>A$ holds. 

Let $99^\prime\equiv 1\mymod{\ell n^2}$ and $\ell n^2 \overline{\ell n^2} \equiv 1 \mymod{9}.$ We have
$$f\equiv r\mymod{\ell n^2}, \  f\equiv \overline{d^2}\mymod{9}
\implies
f\equiv r99^\prime + \overline{d^2\ell n^2}\ell n^2 \mymod{9\ell n^2} $$
by the Chinese Remainder Theorem. Thus, applying Poisson summation over $f$, we obtain
\begin{align*}
    &\sum_{f\equiv \overline{d^2}\mymod 9} \chi_{\ell n^2}(f)\Phi\left(\frac{|fd^2|}{X}\right)
    =
    \sum_{r\mymod{\ell n^2}}\sum_{\substack{f\equiv \overline{d^2}\mymod 9\\f\equiv r\mymod{\ell n^2}}} \chi_{\ell n^2}(f)\Phi(|fd^2|/X)
    \\
    &=
    \frac{X}{|9{\ell n^2}d^2|}\sum_{k\in\mathbb{Z}[\omega]} \widehat{\Phi}\bigg(\sqrt{\frac{X|k|}{|9{\ell n^2}d^2|}}\bigg)
    \sum_{r\mymod {\ell n^2}} \chi_{\ell n^2}(r)
    e\bigg(\frac{-(r9{9}^\prime+\overline{d^2\ell n^2}\ell n^2)k}{{9\ell n^2}\sqrt{-3}}\bigg).
\end{align*}
Substituting this into the first term of \eqref{eq. S_2_chi_bar}, we have
\begin{align}
\label{eq. S_2_chi_bar_afterPS}
\begin{split}
    &\frac{X}{|9|L}
    \sum_{\substack{d\in \mathbb{Z}[\omega]\\ |d|\le A}}\frac{\mu(d)\overline{\chi_d(\ell)}}{|d|^2}\sum_{\substack{n\in \mathbb{Z}[\omega]}}\frac{\Lambda_\mathcal{K}(n)}{\sqrt{|n|}}
	\hat{h}\left(\frac{\log |n|}{L}\right)
	\chi_{d^2}(n^2) \\
    &\times
    \sum_{k\in \mathbb{Z}[\omega]} \widehat{\Phi}\bigg(\sqrt{\frac{X|k|}{|9{\ell n^2}d^2|}}\bigg)
    \frac{1}{|\ell n^2|}\sum_{r\mymod {\ell n^2}} \chi_{\ell n^2}(r)e\bigg(\frac{-(r9{9}^\prime+\overline{d^2\ell n^2}\ell n^2)k}{{9\ell n^2}\sqrt{-3}}\bigg).
    \end{split}
\end{align}

    Similar to $S_2(\chi)$, $k=0$ contributes a main term to $S_2(\overline{\chi})$, in which case the sum over $r$ gives $\phi_\mathcal{K}(\ell n^2)$ when $\ell n^2$ is a cube. Thus, when $k=0,$ we write \eqref{eq. S_2_chi_bar_afterPS} as
\begin{align}\label{eq. S_2_chi_bar after Euler product}
\begin{split}
    &\frac{X\widehat{\Phi}(0)}{|9|L}
	\sum_{\substack{n\in \mathbb{Z}[\omega]}}
 \frac{\Lambda_\mathcal{K}(n)}{\sqrt{|n|}}  
 \frac{\phi_\mathcal{K}(\ell n^2)}{|\ell n^2|} 
\bigg(\prod_{p \nmid n\ell} \left(1-\frac{1 }{|p|^2}\right) +O\left(\frac{1}{A}\right)\bigg)
	\hat{h}\left(\frac{\log |n|}{L}\right) 
    \\
    &=
    \frac{X\widehat{\Phi}(0)}{|9|L}
	\sum_{\substack{n\in \mathbb{Z}[\omega]}}
 \frac{\Lambda_\mathcal{K}(n)}{\sqrt{|n|}}  
 \bigg(\zeta^{-1}_\mathcal{K}(2)  \prod_{p\mid n\ell}\left(1+\frac{1}{|p|}\right)^{-1} +O\left(\frac{1}{A}\right)\bigg)
	\hat{h}\left(\frac{\log |n|}{L}\right).
    \end{split}
\end{align}

When $\ell=a^3$ and $n=q^3$ for a prime power $q$,
we have exactly the same term as in \eqref{eq. main term 1}, 
\begin{align}
\label{eq. chi_bar main term 1}
   \begin{split}
    &\frac{X\widehat{\Phi}(0)}{|9|L}
	\sum_{\substack{q\in \mathbb{Z}[\omega] }}
 \frac{\Lambda_\mathcal{K}(q)}{\sqrt{|q|^3}}  
 \bigg(\zeta^{-1}_\mathcal{K}(2)  \prod_{p\mid q\ell}\left(1+\frac{1}{|p|}\right)^{-1} +O\left(\frac{1}{A}\right)\bigg)
	\hat{h}\left(\frac{3\log |q|}{L}\right)
    \\
    &=
    \frac{X\widehat{\Phi}(0)C_{3, h}}{|9|L}
    \zeta^{-1}_\mathcal{K}(2)  \prod_{p\mid \ell}\left(1+\frac{1}{|p|}\right)^{-1}
	+ O\left(\frac{X}{AL}\right).  
    \end{split}
\end{align}
%where $C_{3,h}$ is defined in \eqref{eq. the c constant}.

When $\ell=p_1a^3$ for a unique prime $p_1$, we have $\ell n^2 = a^3 \implies  n=p_1^{3\alpha+1}$, so \eqref{eq. S_2_chi_bar after Euler product} becomes
\begin{align}
\label{eq. chi_bar main term 2}
\begin{split}
    &\frac{X\widehat{\Phi}(0)\log|p_1|\left(\zeta^{-1}_\mathcal{K}(2)  \prod_{p\mid p_1\ell}(1+\frac{1}{|p|})^{-1}+O\left(\frac{1}{A}\right)\right)}{|9|\sqrt{|p_1|} L}
    \sum_{\alpha=0}^\infty
    \frac{\hat{h}\left(\frac{(3\alpha+1)\log |p_1|}{L}\right)}{|p_1|^{3\alpha/2}} 
    \\
    &\ll
    \frac{X\widehat{\Phi}(0)\log|p_1|\zeta^{-1}_\mathcal{K}(2)  \prod_{p\mid p_1\ell}(1+\frac{1}{|p|})^{-1}}{|9|\sqrt{|p_1|}L}
    \sum_{\alpha=0}^{\frac{L}{3\log|p_1|}}
    \frac{1}{|p_1|^{3\alpha/2}}
    \ll
 \frac{X\widehat{\Phi}(0)\log|p_1|\zeta^{-1}_\mathcal{K}(2)  \prod_{p\mid p_1 \ell }(1+\frac{1}{|p|})^{-1}}{|9|(\sqrt{|p_1|}-|p_1|^{-1})L}.
 \end{split}
\end{align}
When $\ell=p_2^2a^3$ for a unique prime $p_2$, we have $n=p_2^{3\alpha+2}$ and
\begin{align}
\label{eq. chi_bar main term 3}
\begin{split}
    &\frac{X\widehat{\Phi}(0)\log|p_2|\left(\zeta^{-1}_\mathcal{K}(2)  \prod_{p\mid p_2\ell}(1+\frac{1}{|p|})^{-1}+O\left(\frac{1}{A}\right)\right)}{|9 p_2|L}
    \sum_{\alpha=0}^\infty
    \frac{\hat{h}\left(\frac{(3\alpha+2)\log |p_2|}{L}\right)}{|p_2|^{3\alpha/2}} 
    \\
    &\ll
    \frac{X\widehat{\Phi}(0)\log|p_2|\zeta^{-1}_\mathcal{K}(2)  \prod_{p\mid p_2\ell}(1+\frac{1}{|p|})^{-1}}{|9 p_2|L}
    \sum_{\alpha=0}^{\frac{L}{3\log|p_2|}}
    \frac{1}{|p_2|^{3\alpha/2}}
    \ll
 \frac{X\widehat{\Phi}(0)\log|p_2|\zeta^{-1}_\mathcal{K}(2)  \prod_{p\mid p_2 \ell }(1+\frac{1}{|p|})^{-1}}{|9|(|p_2|-|p_2|^{-1/2})L}.
\end{split}
\end{align}
(Note that the roles of $p_1$ and $p_2$ are switched compared to \eqref{eq. main term 2} and \eqref{eq. main term 3}.)

When $k\ne 0$, we first show that the Gauss sum forces $k$ to be divisible by $n$, otherwise the sum vanishes. With a change of variable, we thus sum over $k^*=k/n$, where the $|k^*|\ge \frac{|9\ell d^2 n|}{X}$ contribution is small as before.
By Lemma \ref{lemma Gauss sum=0}, if $p^{2j+a-1} \mid \mid k$, then
\begin{align*}
    \frac{g(k, \ell p^{2j})}{|\ell p^{5j/2}|} 
    &=
    \frac{\chi_{p^{2j+a}}(\ell^\prime)\chi_{\ell^\prime}(p^{2j+a})g(k, p^{2j+a})g(k, \ell^\prime)}{|\ell p^{5j/2}|}
    \\
    &=
    \frac{|p|^{2j+a-1}g(k, \ell^\prime)}{|\ell p^{5j/2}|}
    \times
    \begin{cases}
        -1 &\text{ if } 2j+a\equiv 0 \mymod{3},\\
        \overline{\chi_p(\ell^\prime)}g(k^\prime, p)   &\text{ if }    2j+a\equiv1\mymod{3},\\
        \chi_p(\ell^\prime)\overline{g(k^\prime, p)}   &\text{ if }    2j+a\equiv2\mymod{3},
    \end{cases}
\end{align*}
where, as before, $p^a\mid \mid \ell$ with the corresponding quotients after depleting their powers of $p$ denoted by $\ell^\prime$ and $k^\prime$. If $k$ admits a $p$ power other than $p^{2j+a-1}$, then
\begin{align*}
    \frac{g(k, \ell p^{2j})}{|\ell p^{5j/2}|} 
    &=
    \frac{\chi_{p^{2j+a}}(\ell^\prime)\chi_{\ell^\prime}(p^{2j+a})g(k, p^{2j+a})g(k, \ell^\prime)}{|\ell p^{5j/2}|}
    \\
    &=
    \frac{g(k, \ell^\prime)}{|\ell p^{5j/2}|}
    \times
    \begin{cases}
        \phi_\mathcal{K}(p^{2j+a}) &\text{ if } p^{2j+a}\mid k, 2j+a\equiv 0 \mymod{3},\\
        0 &\text{ otherwise}.
    \end{cases}
\end{align*}
Hence, the sum does not vanish only if $k$ admits the prime power $p^{2j-1}$, which is at least $p^j=n$ for all integers $j\ge 1.$ 
This enables the same reduction to $n=p, p^2,$ with $\ell n^2$ not a cube, and $(n, \ell)=1$, with respective contributions of order at most $Ae^{(\frac{1}{2}+\frac{1}{\alpha}+\epsilon)L}|\ell|/{L}, Ae^{L/2+\epsilon}|\ell|/{L}, |\ell|^\epsilon$ that are all smaller than the final bound in \eqref{eq. final S_2 error bound}.

We show the initial step of the prime case, which presents a difference on the surface compared to $S_2(\chi)$, but can be treated in the same way. We also note that the prime-square case is similarly bounded. When $j=1$, the Gauss sum vanishes unless $n\mid \mid k$, which implies that $k^\prime=k/n=k^*$. 
The sum over $n$ above is
\begin{align*}
%\label{eq. partial S_2 error}
\begin{split}
    &\frac{\chi_\ell(-9\sqrt{-3})g(k, \ell)}{|\ell|}
    \sum_{\substack{n\in \mathbb{Z}[\omega]}} {\chi_{d}( n)}\frac{\Lambda_\mathcal{K}(n)}{\sqrt{|n|^3}}
	\hat{h}\left(\frac{\log |n|}{L}\right)
	\widehat{\Phi}\bigg(\sqrt{\frac{X|k^\prime|}{|9{\ell n}d^2|}}\bigg)
    e\bigg(\frac{-\overline{d^2n^2\ell}k}{9\sqrt{-3}}\bigg)
    \overline{\chi_{n}(-9\sqrt{-3})g(k^\prime \ell, n)}
    \\
    &\ll
    \sum_{\substack{n\in \mathbb{Z}[\omega]\\}} \frac{\Lambda_\mathcal{K}(n)}{\sqrt{|n|^3}}
	\hat{h}\left(\frac{\log |n|}{L}\right)
	\widehat{\Phi}\bigg(\sqrt{\frac{X|k^\prime|}{|9{\ell n}d^2|}}\bigg)
    e\bigg(\frac{-\overline{d^2n^2\ell}k}{9\sqrt{-3}}\bigg)
    \overline{\chi_{n}(-9\sqrt{-3})
    g(dk^\prime \ell, n)},
 \end{split}   
\end{align*}
where the complex conjugate does not change the bound, and the outer sum runs over $k^\prime$. Following the procedure in bounding the analogous term in $S_2(\chi)$, we obtain that the errors of $S_2({\overline\chi})$ are at most $o(X)$ for
\begin{align}
\label{eq. final S_2_chi error}
    e^{11L/27}|\ell|^{14/27}\le X^{13/27}.
\end{align}

Since $S_2=S_2(\chi)+S_2(\overline{\chi}),$  putting \eqref{eq. main term 1}, \eqref{eq. main term 2}, \eqref{eq. main term 3}, and \eqref{eq. final S_2 error bound} of $S_2(\chi)$, and \eqref{eq. chi_bar main term 1}, \eqref{eq. chi_bar main term 2}, \eqref{eq. chi_bar main term 3}, and \eqref{eq. final S_2_chi error} of $S_2(\overline{\chi})$ together, we derive the statements in the theorem.
\end{proof}

\subsubsection{Completing the proof of Proposition \ref{Prop 3}}

Finally, we shall prove Proposition \ref{Prop 3} by applying Theorems \ref{theorem S_1} and \ref{theorem S_2}. (We note that by taking $\ell=1$, this recovers the result in \cite{DG} with all prime powers included.)

\begin{proof}
By Theorem \ref{theorem S_1}, when $\ell$ is a cube, a straightforward calculation yields
\begin{align*}
    S_1=\frac{X\widehat{\Phi}(0)\hat{h}(0)}{|9|L}  \zeta^{-1}_\mathcal{K}(2)  \prod_{p\mid \ell}\left(1+\frac{1}{|p|}\right)^{-1} (\log X +O(1))
    +O\bigg(\frac{X^{1/2+\epsilon}}{L}\bigg),
\end{align*}
and $S_1\ll X^{1/2+\epsilon}/{L}$ otherwise.
By this and Theorem \ref{theorem S_2} for $S_2$, when $\ell$ is a cube and $e^{\frac{11L}{27}}|\ell|^{\frac{14}{27}}\le X^{\frac{13}{27}}$,
\begin{align*}
    \sum_{\mathfrak{f}\in\mathcal{F}} \sum_{\gamma_{\mathfrak{f}}} h\left( \frac{\gamma_{\mathfrak{f}} L}{2\pi}\right)\chi_\mathfrak{f}(\ell) \Phi\bigg(\frac{|\mathfrak{f}|}{X}\bigg)
    &=
 \frac{X\widehat{\Phi}(0)\hat{h}(0)}{|9|L}  \zeta^{-1}_\mathcal{K}(2)  \prod_{p\mid \ell}\left(1+\frac{1}{|p|}\right)^{-1} (\log X +O(1))
    +O\bigg(\frac{X^{1/2+\epsilon}}{L}\bigg)
    \\
    &+
    \frac{X\widehat{\Phi}(0)C_{3,h}\zeta^{-1}_\mathcal{K}(2)  \prod_{p\mid \ell}(1+\frac{1}{|p|})^{-1}}{|9|L}
    +O\bigg(\frac{X^{\frac{14}{27}+\epsilon}e^{\frac{11L}{27}}|\ell|^{\frac{14}{27}}}{L} \bigg),
\end{align*}
 which equals
 \begin{align*}
 \frac{X\widehat{\Phi}(0)}{|9|L}  \zeta^{-1}_\mathcal{K}(2)  \prod_{p\mid \ell}\bigg(1+\frac{1}{|p|}\bigg)^{-1} 
 \left(\hat{h}(0) \log X + C_{3,h}+O(1)\right) 
 + O\bigg( \frac{X^{\frac{14}{27}+\epsilon}e^{\frac{11L}{27}}|\ell|^{\frac{14}{27}}}{L} \bigg).
\end{align*}
On the other hand, when $\ell=qa^3$  or $\ell= q^2a^3$ for a unique prime $q$, we have
\begin{align*}
  D^T(X;\ell,h, \Phi)
    \ll
     \frac{X\widehat{\Phi}(0)\log|q|(1+\sqrt{|q|})\zeta^{-1}_\mathcal{K}(2)  \prod_{p\mid q \ell }(1+\frac{1}{|p|})^{-1}}{|9|(|q|-|q|^{-1/2})L} 
    + 
    \frac{X^{\frac{14}{27}+\epsilon}e^{\frac{11L}{27}}|\ell|^{\frac{14}{27}}}{L}.
\end{align*}
Lastly, when neither is the case, it is clear that
    $$D^T(X;\ell,h, \Phi)
    \ll
    \frac{X^{\frac{14}{27}+\epsilon}e^{\frac{11L}{27}}|\ell|^{\frac{14}{27}}}{L},$$
    which concludes the proof.
\end{proof}

\section*{Acknowledgments}
We thank Maksym Radziwi\l\l \ for numerous conversations and helpful comments about this project.

\end{document}